\documentclass{amsart}

\usepackage{amsmath,amssymb,amsthm,amsfonts,mathrsfs}
\usepackage[frame,cmtip,arrow,matrix,line,graph,curve]{xy}
\usepackage{graphpap,color,paralist}
\usepackage[mathscr]{eucal}
\usepackage{mathabx}
\usepackage[pdftex,colorlinks,backref=page,citecolor=blue]{hyperref}
\usepackage{tikz}
\usepackage{epic,eepic}
\usepackage{yfonts}
\usepackage{enumerate}
\usepackage{comment}
\usepackage{adjustbox}
\usepackage[smalltableaux]{ytableau}
\usetikzlibrary{patterns.meta}

\theoremstyle{definition}
\newtheorem{theorem}{Theorem}[section]
\newtheorem{definition}[theorem]{Definition}
\newtheorem{conjecture}[theorem]{Conjecture}
\newtheorem{lemma}[theorem]{Lemma}
\newtheorem{proposition}[theorem]{Proposition}
\newtheorem{corollary}[theorem]{Corollary}

\newtheorem{question}[theorem]{Question}

\theoremstyle{remark}
\newtheorem{remark}[theorem]{Remark}
\newtheorem{example}[theorem]{Example}
\def\PP{\mathbb{P}}
\def\RR{\mathbb{R}}
\def\CC{\mathbb{C}}
\def\ZZ{\mathbb{Z}}
\def\QQ{\mathbb{Q}}
\def\kk{\mathbb{K}}
\def\T2{\mathbb{T}_2}
\def\Mat{\text{Mat}}
\def\H{\mathrm{H}}
\def\N{\mathrm{N}}
\def\MV{\text{MV}}

\DeclareMathOperator{\rk}{rank}

\DeclareMathOperator{\NS}{NS}
\DeclareMathOperator{\Corr}{Corr}

\DeclareMathOperator{\Nef}{Nef}
\DeclareMathOperator{\Eff}{Eff}

\newcommand{\medwedge}{\mathord{\adjustbox{valign=B,totalheight=.45\baselineskip}{$\bigwedge$}}}

\begin{document}

\title{
Realizations of homology classes and projection areas  
}

\author{Daoji Huang}
\address{Institute for Advanced Study and University of Massachusetts Amherst}
\email{daojihuang@umass.edu}

\author{June Huh}
\address{Princeton University and Korea Institute for Advanced Study
}
\email{huh@princeton.edu}

\author{Mateusz Micha\l ek}
\address{
University of Konstanz
}
\email{mateusz.michalek@uni-konstanz.de}

\author{Botong Wang}
\address{University of Wisconsin–Madison
}
\email{wang@math.wisc.edu}

\author{Shouda Wang}
\address{Princeton University}
\email{shoudawang@princeton.edu}

\begin{abstract}
The relationship between convex geometry and algebraic geometry has deep historical roots, tracing back to classical works in enumerative geometry.
In this paper, we continue this theme by studying two interconnected problems regarding projections of geometric objects in four-dimensional spaces:
\begin{enumerate}[(1)]
\item Let $A$ be a convex body in $\mathbb{R}^4$, and let $(p_{12}, p_{13}, p_{14}, p_{23}, p_{24}, p_{34})$ be the areas of the six coordinate projections of $A$ in $\mathbb{R}^2$. Which tuples of six  nonnegative real numbers can arise in this way?
\item Let $S$ be an irreducible surface in $(\mathbb{P}^1)^4$, and let $(p_{12}, p_{13}, p_{14}, p_{23}, p_{24}, p_{34})$ be the degrees of the six coordinate projections from $S$ to $(\mathbb{P}^1)^2$. Which tuples of six nonnegative integers can arise in this way?
\end{enumerate}
We show that these questions are governed by the Pl\"ucker relations for the Grassmannian $\text{Gr}(2,4)$ over the triangular hyperfield $\mathbb{T}_2$. We extend our analysis by determining
 the homology classes in $(\mathbb{P}^m)^n$  proportional to the fundamental classes of irreducible algebraic surfaces, resolving the algebraic Steenrod problem in this setting.
Our results lead to several conjectures on realizable homology classes in smooth projective varieties and on the projection volumes of convex bodies.
\end{abstract}

\maketitle

\section{Introduction}

The connection between convex geometry and intersection theory traces back to classical results of Newton \cite{Newton} and Minding \cite{Minding} that relate the solutions of polynomial systems to their Newton polyhedra.\footnote{For historical overviews and commentaries on Newton's and Minding's work from a modern perspective, see \cite[Essay 4.4]{Edwards} and \cite[Section 2]{Cox-Rojas}.}
In this paper, we expand upon this classical theme by addressing two closely related problems: one concerning projection areas of convex bodies in $\mathbb{R}^4$, and the other involving homology classes of algebraic surfaces in $(\mathbb{P}^1)^4$, drawing inspiration from both convex and algebraic geometry.

\begin{question}
\label{q:convexbody}
Let $A$ be a convex body in $\mathbb{R}^4$, and let $\mathbf{p}=(p_{12}, p_{13}, p_{14}, p_{23}, p_{24}, p_{34})$ be the areas of the six coordinate projections of $A$ in $\mathbb{R}^2$. Which tuples of six  nonnegative real numbers can arise in this way?
\end{question}

\begin{question}
\label{q:realizablity}
Let $S$ be an irreducible surface in $(\mathbb{P}^1)^4$, and let $\mathbf{p}=(p_{12}, p_{13}, p_{14}, p_{23}, p_{24}, p_{34})$ be the degrees of the six coordinate projections from $S$ to $(\mathbb{P}^1)^2$. Which tuples of six nonnegative integers can arise in this way?
\end{question}

 
 We answer Question~\ref{q:convexbody} 
 in Theorem~\ref{thm:convexbody}, and Question~\ref{q:realizablity} modulo positive multiples in Theorem~\ref{thm:mainQ}. Based on our findings, we formulate a possible answer to the $n$-dimensional analogue of Question~\ref{q:convexbody} in Conjecture~\ref{con:realizabilitybypolytopes}, and provide an extension of Theorem~\ref{thm:mainQ} from $(\mathbb{P}^1)^4$ to $(\mathbb{P}^n)^m$ in Theorem~\ref{thm:nondegenerateSteenrod}.

Questions ~\ref{q:convexbody} and ~\ref{q:realizablity} find their solutions encoded in the Pl\"ucker relations for the Grassmannian $\text{Gr}(2,4)$ over a triangular hyperfield $\mathbb{T}_2$.
For general discussions of Pl\"ucker relations over triangular hyperfields and other algebraic objects, we refer to \cite{Baker-Bowler,BHKL}. For our purposes, the following explicit definition will suffice.

\begin{definition}
Let $\Delta(\mathbb{T}_2)$ be the set of $(p_{12},p_{13},p_{14},p_{23},p_{24},p_{34}) \in \medwedge^2\,\RR^4$  such that
\[
p_{ij} \ge 0 \ \ \text{and} \ \ 
\sqrt{p_{ij}p_{kl}}+ \sqrt{p_{ik}p_{jl}} \ge \sqrt{p_{il}p_{jk}} \ \  \text{for any $i,j,k,l$.}
\]
\end{definition}

The set $\Delta(\mathbb{T}_2)$ is a non-convex semialgebraic set 
whose image in the projective space of $\medwedge^2\, \RR^4$ is homeomorphic to the $5$-dimensional closed ball \cite{BHKL}.
This image is the set  of $\mathbb{T}_2$-valued points in the Grassmannian $\text{Gr}(2,4)$ in the sense of \cite{Baker-Lorscheid}, so one may view  $\Delta(\mathbb{T}_2)$ as the set of $\mathbb{T}_2$-valued points in the affine cone over the Grassmannian.
The points in the interior $\Delta(\mathbb{T}_2)^\circ$ are called \emph{nondegenerate}. These are the sextuples of positive numbers such that the square roots  of $p_{12}p_{34}$, $p_{13}p_{24}$, $p_{14}p_{23}$ form the side lengths of a nondegenerate triangle. 



We answer Question~\ref{q:convexbody}. Let $\pi_{ij}$ be the coordinate projection of $\mathbb{R}^4$ onto the plane orthogonal to the standard basis vectors $\mathbf{e}_i$ and $\mathbf{e}_j$.

\begin{theorem}\label{thm:convexbody}
The following holds for any $\mathbf{p}=(p_{12},p_{13},p_{14},p_{23},p_{24},p_{34}) \in \medwedge^2\,\RR^4$.
\begin{enumerate}[(1)]\itemsep 5pt
\item $\mathbf{p} \in \Delta(\T2)^\circ$ if and only if there is a smooth convex body $A\subseteq \mathbb{R}^4$ that satisfies
\[
p_{ij}=\Big(\text{the area of the projection $\pi_{ij}(A)$}\Big) \ \ \text{for all $i< j$.}
\]
\item $\mathbf{p} \in \Delta(\T2)$ if and only if
 there is a convex body $A\subseteq \mathbb{R}^4$ that satisfies
\[
p_{ij}=\Big(\text{the area of the projection $\pi_{ij}(A)$}\Big) \ \ \text{for all $i< j$.}
\]
\end{enumerate}
\end{theorem}

 For any undefined terms in the theory of convex bodies, we refer to \cite{Schneider}. One of our key tools will be the theory of mixed volumes \cite[Chapter 5]{Schneider}.
To each pair of convex bodies $A$ and $B$ in $\RR^4$, we associate a vector of real numbers 
\[
A\wedge B=(p_{12},p_{13},p_{14},p_{23},p_{24},p_{34}) \in \medwedge^2\, \RR^4,
\]
where the components $p_{ij}$ are defined as the mixed volumes of the projections of $A$ and $B$:
\[
p_{ij}=\Big(\text{the mixed volume of $\pi_{ij}(A)$ and $\pi_{ij}(B)$ in $\RR^2$}\Big).
\]
With this notation, the basic properties of mixed volumes read
\[
A \wedge B=B\wedge A  \ \ \text{and} \ \  \quad (\lambda_1 A_1+\lambda_2 A_2)\wedge B=\lambda_1( A_1 \wedge B )+\lambda_2(A_2\wedge B) \ \ \text{for $\lambda_1,\lambda_2 \ge 0$.}
\]
Theorem~\ref{thm:convexbody} provides a characterization of vectors of the form $A \wedge A$. 
Our second main contribution is 
Theorem~\ref{thm:mainQpoly}, which states that 
\[
\Delta(\mathbb{T}_2) \cap \medwedge^2\, \QQ^4=\left\{A \wedge B\;\middle|\; \text{$A,B$ are rational polytopes in $\RR^4$}\right\}.
\]
We conjecture that, in fact, $\Delta(\mathbb{T}_2) \cap \medwedge^2\, \QQ^4$ is the set of projection areas of rational polytopes in $\RR^4$.

The parallel Question \ref{q:realizablity} in algebraic geometry can be rephrased as follows: For $E \subseteq [4]$, write $\varphi_E$ for the coordinate projection $(\mathbb{P}^1)^{4} \to (\mathbb{P}^1)^{4-|E|}$ that forgets the coordinates labelled by $E$. 
If $S$ is an irreducible surface in $(\mathbb{P}^1)^4$ over the complex numbers, we can uniquely express its homology class as a nonnegative integral linear combination
\[
[S]=p_{12}[\PP^1 \times \PP^1 \times \PP^0 \times \PP^0]
+\cdots +p_{34}[\PP^0 \times \PP^0 \times \PP^1 \times \PP^1].
\] 
Which vectors of nonnegative integers $(p_{12},p_{13},p_{14},p_{23},p_{24},p_{34})$ can arise in this way?
We call such homology classes \emph{realizable}. Such questions are algebraic analogues of the \emph{Steenrod problem} in topology \cite[Problem 25]{Eilenberg}, 
which asks whether every homology class in any simplicial complex $X$ is the image of the fundamental class of a closed oriented manifold by a map into the simplicial complex. In 1954, as part of his work on cobordism theory, Thom showed that in general the answer to this question depends on the chosen coeﬃcients \cite{Thom}: the answer is positive for $\H_\bullet(X, \mathbb{Q})$ and negative for $\H_\bullet(X, \mathbb{Z})$. 
We introduce the necessary definitions adapted to the setting of algebraic geometry.

\begin{definition}
    Let $X$ be a complex smooth projective variety.\footnote{In this paper, a variety and its subvarieties are by definition reduced and irreducible.}
    \begin{enumerate}[(1)]\itemsep 5pt
        \item A class $\eta \in \text{H}_{2k} (X,\mathbb{Z})$ is \emph{realizable over $\mathbb{Z}$} if there is
        a subvariety $V\subseteq X$ with $\eta =[V]$.
        \item A class $\eta \in \text{H}_{2k}(X,\mathbb{Q})$ is \emph{realizable over $\mathbb{Q}$} if  $\lambda \eta$ is realizable over $\mathbb{Z}$ for some $\lambda \in \mathbb{Q}_{\ge 0}$.
        \item A class $\eta\in \text{H}_{2k}(X,\mathbb{R})$
    is \emph{realizable over $\mathbb{R}$}
    if it is a limit of classes realizable
    over $\mathbb{Q}$.
    \end{enumerate}
\end{definition}


We provide a characterization of homology classes in $(\mathbb{P}^1)^4$ that are realizable over $\mathbb{Q}$. 

\begin{theorem}\label{thm:mainQ}
For $\mathbf{p}=(p_{12},p_{13},p_{14},p_{23},p_{24},p_{34}) \in \medwedge^2\,\QQ^4$, consider the homology class
\[
\eta(\mathbf{p}) \coloneq p_{12}[\PP^1 \times \PP^1 \times \PP^0 \times \PP^0] 
+\cdots +p_{34}[\PP^0 \times \PP^0 \times \PP^1 \times \PP^1] \in \H_4((\mathbb{P}^1)^4,\mathbb{Q}).
\]
The class $\eta(\mathbf{p})$ is realizable over $\mathbb{Q}$ if and only if $\mathbf{p}$ is in the set $\Delta(\mathbb{T}_2)$.
\end{theorem}

Compare Theorem~\ref{thm:mainQ} with Theorem~\ref{thm:CI}, where we identify the set of complete intersection surface classes in $(\PP^1)^4$, up to a rational multiple, with the set of rational points in
\[
\Delta(\mathbb{T}_1) 
\coloneq \left\{
\,\mathbf{p} \in \medwedge^2\,\RR^4  \;\middle|\; 
p_{ij} \ge 0 \ \ \text{and} \ \ p_{ij} p_{kl} + p_{ik} p_{jl} \ge p_{il} p_{jk} \ \ \text{for any $i,j,k,l$}
\,\right\}.
\]
As in the case of $\Delta(\mathbb{T}_2)$, the image of $\Delta(\mathbb{T}_1)$ in the projective space of $\medwedge^2\, \RR^4$ is the set of $\mathbb{T}_1$-valued points in the Grassmannian $\text{Gr}(2,4)$, also homeomorphic to the $5$-dimensional closed ball \cite{BHKL}.
A boundary point of $\Delta(\mathbb{T}_2)$ is in $\Delta(\mathbb{T}_1)$ if and only if $p_{ij}=0$ for some $i$ and $j$.

Question \ref{q:realizablity}, the realizability problem over $\mathbb{Z}$ for $(\mathbb{P}^1)^4$, is more subtle.
For example, the homology class corresponding to $\mathbf{p}=(1,1,1,1,1,3)$ is realizable over $\mathbb{Q}$ but not over $\mathbb{Z}$.  
To see this, note that the hypothetical surface $S$ should satisfy
\[
[\varphi_{3}(S)]=[\varphi_{4}(S)]=[\mathbb{P}^1 \times \mathbb{P}^1 \times \mathbb{P}^0]+[\mathbb{P}^1 \times \mathbb{P}^0 \times \mathbb{P}^1]+[\mathbb{P}^0 \times \mathbb{P}^1 \times \mathbb{P}^1].
\]
Thus, the defining equations of  $\varphi_{3}^{-1} \varphi_{3}(S)$ and $\varphi_{4}^{-1} \varphi_{4}(S)$ in an affine chart
are of the form 
\begin{align*}
*1 + *x_1+  *x_2+ *x_4+ *x_1x_2+ *x_1x_4+ *x_2x_4+*x_1x_2x_4 &=0,\\
   *1 + *x_1+  *x_2+ *x_3+ *x_1x_2+ *x_1x_3+ *x_2x_3+*x_1x_2x_3 &=0, 
\end{align*}
where the $*$'s are placeholders for coefficients. 
For generic values of $x_3$ and $x_4$,
this system has at most 2 solutions, contradicting $p_{34}=3$.
In Proposition~\ref{cor:necconditions}, we observe more generally that the realizability of $\mathbf{p}$ over $\mathbb{Z}$ implies 
\[ p_{ij}\le p_{ik}p_{jl}+p_{il}p_{jk} \ \ \text{for any $i,j,k,l$.} 
\]
For $\mathbf{p} \in \Delta(\mathbb{T}_2)^\circ$, we know no other obstructions to the realizability of $\mathbf{p}$ over $\mathbb{Z}$.
In Section~\ref{sec:manyProjectiveSpaces}, 
using global surjectivity of the period map for marked complex K3 surfaces \cite[Chapter 7]{Huybrechts},  
we show that for any integral vector $\mathbf{p} \in \Delta(\mathbb{T}_2)$ and any integer $\lambda_1>1$, there is an integer $\lambda_2>0$ such that $\lambda_1\lambda_2^{-1}\eta(\mathbf{p})$ is realizable over $\ZZ$.


\begin{definition}
We say that a real symmetric matrix is \emph{Lorentzian} if it has only nonnegative entries and has at most one positive eigenvalue.
\end{definition}

Let $\mathbf{m}$ be a vector of positive integers $(m_1,\ldots,m_n)$, and let  $\mathbb{P}^\mathbf{m}$ be the product of projective spaces $\prod_{i=1}^n \mathbb{P}^{m_i}$.  
Let $H_i$ be the pullback of the cohomology class of a hyperplane by the $i$-th projection $\mathbb{P}^\mathbf{m} \to \mathbb{P}^{m_i}$.
As observed in \cite{BHKL}, a point $\mathbf{p} \in \medwedge^2\, \RR^4$ is in $\Delta(\mathbb{T}_2)$ if and only if 
\[  
L(\mathbf{p})\coloneq\begin{pmatrix}
0&p_{12}&p_{13}&p_{14}\\
p_{12}&0&p_{23}&p_{24}\\
p_{13}&p_{23}&0&p_{34}\\
p_{14}&p_{24}&p_{34}&0
\end{pmatrix}
\]
is a Lorentzian matrix, see Proposition~\ref{prop:T2}.
Therefore, the following statement extends Theorem~\ref{thm:mainQ}.

\begin{theorem}\label{thm:nondegenerateSteenrod}
For $\eta \in \H_4(\mathbb{P}^\mathbf{m},\mathbb{Q})$, consider the $n \times n$ symmetric matrix $L(\eta)$ with entries
 \[
L(\eta)_{ij}=\int_\eta H_iH_j \ \ \text{for $1 \le i \le j \le n$.}
 \]
 The class $\eta$ is realizable over $\mathbb{Q}$ if and only if  $L(\eta)$ is Lorentzian.
\end{theorem}

Extending Theorem~\ref{thm:nondegenerateSteenrod},
we propose in Conjecture~\ref{con:main} a numerical characterization of $2$-dimensional universally pseudoeffective classes  that are realizable over $\mathbb{R}$. 
For the case when $X$ is the Grassmannian $\text{Gr}(d,n)$, see Theorem~\ref{thm:Grassmannian}.

Based on our findings, we formulate a convex geometric analogue of Theorem~\ref{thm:nondegenerateSteenrod} that extends  Theorem~\ref{thm:convexbody} to any dimension $d \ge 3$. Extending the notation before, we write $\pi_{ij}$ for the coordinate projection of $\mathbb{R}^d$ onto the coordinate subspace orthogonal to the standard basis vectors $\mathbf{e}_i$ and $\mathbf{e}_j$.

\begin{conjecture}\label{con:realizabilitybypolytopes}
For any $d \times d$ real symmetric matrix $(p_{ij})$ with nonnegative off-diagonal and zero diagonal entries, the following conditions are equivalent.
\begin{enumerate}[(1)]\itemsep 5pt
\item There is a convex body $A \subseteq \mathbb{R}^d$ that satisfies
\[
p_{ij}=\Big(\text{the volume of the projection $\pi_{ij}(A)$}\Big) \ \ \text{for all $i\neq j$.}
\]
\item The matrix $(p_{ij})$ is Lorentzian.
\end{enumerate}
\end{conjecture}

By Corollary~\ref{cor:projectionvolumes}, condition $(1)$ implies condition $(2)$ in Conjecture~\ref{con:realizabilitybypolytopes}. 
By Theorem~\ref{thm:convexbody}, condition $(2)$ implies condition $(1)$ when $d=4$.
One may also formulate the weaker conjecture that $(p_{ij})$ is Lorentzian if and only if there exist convex bodies $A_1,\ldots,A_{d-2} \subseteq \RR^d$ such that
\[
p_{ij}=\Big(\text{the mixed volume of the projections $\pi_{ij}(A_1),\ldots,\pi_{ij}(A_{d-2})$}\Big) \ \ \text{for all $i\neq j$.}
\]
See Remark~\ref{rmk:subtlety} 
for the subtlety involved in formulating a similar conjecture for codimension $3$ projections $\pi_{ijk}:\RR^d \to \RR^{d-3}$. 






\subsection*{Notations}



The field $\mathbf{k}$ will be the field of real or rational numbers. We will work with projective varieties over $\CC$, and use singular homology and cohomology of their underlying topological spaces. 
All of our theorems remain valid over an arbitrary algebraically closed field if we replace the homology group by the group of algebraic cycles modulo homological or numerical equivalence, except in Theorem~\ref{prop:exists surface}, where we need the ground field to be uncountable or of characteristic $0$,  and in Theorem~\ref{thm:realization over Z}, where we need the ground field to be of characteristic $0$. 
A lattice polytope in $\RR^d$ is a convex polytope all of whose vertices are in $\ZZ^d$.
A rational polytope in $\RR^d$ is a convex polytope all of whose vertices are in $\QQ^d$.



\subsection*{Acknowledgements}
The authors thank the Institute for Advanced Study for providing an excellent working environment, and Matt Larson, Elizabeth Pratt, Stefan Schreieder, and Chenyang Xu for their insightful comments. 
Daoji Huang is  supported by the Charles Simonyi Endowment and NSF-DMS2202900.
June Huh is partially supported by the Oswald Veblen Fund, the Fund for Mathematics, and the Simons Investigator Grant.
 Mateusz Micha\l ek  is partially supported by the Charles Simonyi Endowment and the DFG grant 467575307. 
 Botong Wang is partially supported by the NSF grant DMS-1926686.

\section{Homology classes and volumes}\label{sec:homclandvol}


Let $X$ be a complex smooth projective variety, $S$ be an irreducible surface in $X$, and let $D_1,\dots, D_n$ be the cohomology classes of divisors on $X$.\footnote{The statements of this section hold over any algebraically closed field if we use the Chow groups in place of the singular homology groups.} 
We consider the $n \times n$ intersection matrix $L(S)=L(S; D_1,\ldots,D_n)$ defined by
\[
L(S)_{ij}=\int_S D_iD_j.
\]
The following proposition is a direct consequence of the Hodge Index Theorem \cite[Section 5.1]{Hartshorne} applied to the resolution of singularities of $S$.

\begin{proposition}\label{lem:at-most-one-positive-eigenvalue}
    The matrix $L(S)$ has at most one positive eigenvalue. If the linear span of $D_i$ contains an ample class of $X$, then $L(S)$ has exactly one positive eigenvalue.
\end{proposition}


Let $\mathbf{m}$ be a vector of nonnegative integers $(m_1,\ldots,m_n)$, and let  $\mathbb{P}^\mathbf{m}$ be the product of projective spaces $\prod_{i=1}^n \mathbb{P}^{m_i}$.
By the K\"unneth formula, we have
\[
\H^\bullet(\PP^{\mathbf{m}}, \ZZ)\simeq \bigotimes_{i=1}^n \mathbb{Z}[H_i]/(H_i^{m_i+1}), 
\]
where $H_i$ is the pullback of the hyperplane class by the $i$-th projection $\mathbb{P}^{\mathbf{m}} \to \mathbb{P}^{m_i}$.

\begin{corollary}\label{cor:oneposeigen}
For any irreducible surface $S$ in $\mathbb{P}^\mathbf{m}$, the matrix $L(S; H_1,\dots, H_n)$ is Lorentzian. 
\end{corollary}

This corollary provides necessary restrictions on the classes in $\H_4(\PP^{\mathbf{m}}, \ZZ)$ that are realizable over $\ZZ$. However, these conditions are not sufficient, see Proposition~\ref{cor:necconditions} and Example \ref{exm:nonrepbut multiple is}. 

One of our key tools will be the theory of mixed volumes. For a sequence of convex bodies $A_1,\ldots,A_d$ in $\mathbb{R}^d$, the \emph{mixed volume} $\MV(A_1,\ldots,A_d)$ is given by 
\[
\MV(A_1,\ldots,A_d)\coloneq \frac{1}{d!} \frac{\partial}{\partial x_1} \cdots \frac{\partial}{\partial x_d} \text{Vol}(x_1A_1+\cdots+x_dA_d),
\]
where $\text{Vol}$ denotes the $d$-dimensional normalized volume in $\RR^d$. 
The mixed volume enjoys several important properties: it is symmetric in its arguments, multilinear with respect to Minkowski sums with nonnegative coefficients, monotone with respect to inclusion of convex bodies, and satisfies the identity
\[
\MV(A,\ldots,A)=\text{Vol}(A) \ \ \text{for any convex body $A \subseteq \RR^d$.}
\]
For additional background and further properties of mixed volumes, we refer the reader to \cite[Chapter 5]{Schneider}.

The connection to intersection theory is provided by toric geometry. We start by recalling the Bernstein--Khovanskii--Kushnirenko theorem \cite[Section 7.5]{CLO}. Consider an algebraic torus $T\simeq (\CC^*)^d$ with the lattice of characters $M_T\simeq \mathbb{Z}^d$. Let $P_1,\ldots,P_d$ be a sequence of lattice polytopes in $M_T\otimes\mathbb{R}$. 
We write $f_{P_i}$ for a general $\mathbb{C}$-linear combination of characters in $P_i\cap M_T$, viewed as a function on $T$.

\begin{theorem}[Bernstein--Khovanskii--Kushnirenko]\label{thm:BK}
The number of solutions $x\in T$ of the system of equations $f_{P_1}(x)=\ldots=f_{P_d}(x)=0$ is equal to the mixed volume $\MV(P_1,\dots,P_d)$.
\end{theorem}




Let $\mathbf{m}$ be a vector of positive integers $(m_1,\ldots,m_n)$ whose sum is $d \ge 2$. 
In the remainder of this section, let $T$ be the product torus  $\prod_{i=1}^n (\mathbb{C}^*)^{m_i}$ with the lattice of characters $\bigoplus_{i=1}^n \mathbb{Z}^{m_i}$.
For each $i$, we set 
\[
\Delta_i\coloneq (\text{the convex hull of $\mathbf{0}$ and the standard basis vectors of $\mathbb{R}^{m_i}$}) \subseteq \RR^d=\bigoplus_{i=1}^n \mathbb{R}^{m_i}.
\]
We may view $f_{\Delta_i}$ as a function on the torus $T$ of $\mathbb{P}^\mathbf{m}$, whose homogenization defines the inverse image  of a hyperplane under the $i$-th projection $\mathbb{P}^{\mathbf{m}} \to \mathbb{P}^{m_i}$.

\begin{theorem}\label{lem:polytopesgivesurface}
Let $P_1,\ldots,P_{d-2}$ be convex polytopes in $\mathbb{R}^d$. 
\begin{enumerate}[(1)]\itemsep 5pt
\item If every $P_i$ is a rational polytope, then there is  $\lambda \in \QQ_{>0}$ and an irreducible surface $S \subseteq \PP^\mathbf{m}$ such that
\[
\int_S H_iH_j=\lambda^{-1} \, \MV(P_1,\ldots,P_{d-2},\Delta_i,\Delta_j) \ \ \text{for all $i,j$.}
\]
\item If every $P_i$ is a lattice polytope, then there is $\lambda \in \ZZ_{>0}$ and an irreducible surface $S \subseteq \PP^\mathbf{m}$ such that
\[
\int_S H_iH_j=\lambda^{-1} \, \MV(P_1,\ldots,P_{d-2},\Delta_i,\Delta_j) \ \ \text{for all $i,j$.}
\]
\end{enumerate}
In the latter case, if 
$\text{dim}(\sum_{i\in I}P_i)>|I|$ for all nonempty $I \subseteq [d-2]$,
then we may take $\lambda=1$.
\end{theorem}

By Poincar\'e duality, the constant and the mixed volumes in Theorem~\ref{lem:polytopesgivesurface} determine the homology class of $S$ in $\PP^\mathbf{m}$. By varying the polytopes, we obtain many classes in $\H_4(\PP^\mathbf{m},\QQ)$ that are realizable over $\QQ$. In Section~\ref{Section4}, we show that every class in $\H_4((\PP^1)^4,\QQ)$ that is realizable over $\QQ$ can be realized in this way.

\begin{proof}
We use the language and basic constructions of toric geometry. For background, see \cite{FultonToric}. 
To keep our argument valid in arbitrary characteristic,  we avoid applying Bertini’s theorem on the smoothness of general members of a basepoint-free linear system.

It is enough to prove the statements on lattice polytopes. 
Let $Y$ be a smooth projective toric variety of $T$ whose normal fan refines the normal fans of $P_i$ and $\Delta_i$.
 Each $P_i$ defines a basepoint-free divisor class $D_i$ on $Y$ and a map
 \[
 f_i:Y\longrightarrow \PP \H^0(Y,\mathcal{O}_Y(D_i))^\vee,
 \]
 whose image has dimension equal to that of $P_i$.
 To simplify notation, we write $D_i$ for a general member in its linear system.
 First suppose that 
 $\dim(\sum_{i\in I}P_i)>|I|$ for every nonempty index set $I$. 
 Under this assumption, we prove  that the intersection $Z$ of all the $D_i$'s is 
 an irreducible surface in $Y$.

We proceed inductively, assuming that the intersection $Z_i:=D_1\cap\dots\cap D_i$ is  
irreducible of codimension $i$ in $Y$ for $i<k$. 
We claim that the dimension of $f_{i+1}(Z_i)$ is strictly larger than one. 

For contradiction, suppose this is not the case. 
Then, the intersection $Z_i \cap D_{i+1}\cap E_{i+1}$ is empty, where $E_{i+1}$ is another general member in the linear system of $D_{i+1}$. Let $H$ be a very ample divisor class on $Y$ corresponding to a polytope $P$. 
The Bernstein--Khovanskii--Kushnirenko theorem implies that 
\[
\MV(P_1,\dots,P_i, P_{i+1}, P_{i+1},P, \dots, P)=0.
\]
By translating $P_i$ and $P$, we may assume they all contain the origin of $\RR^d$. By our assumption and Rado's theorem \cite[Theorem 1]{rado1942theorem}, picking general points in $P_1,\dots, P_i, P_{i+1}, P_{i+1},P,\ldots,P$, we obtain a basis of $\RR^d$. This contradicts the fact that mixed volume is zero, as mixed volume of linearly independent line segments is nonzero and mixed volume is monotone under inclusions of convex bodies.
This finishes the proof of the claim.

We conclude from Bertini's irreducibility theorem \cite[Theorem 6.3 (4)]{jouanolou1983} that $Z_{i+1}$ is irreducible. This finishes the proof that $Z=Z_{d-2}$ is nonempty and irreducible. Writing $S$ for the image of $Z$ in $\PP^\mathbf{m}$, Theorem~\ref{thm:BK} guarantees that
\[
\int_S H_iH_j=\int_Z H_iH_j=\MV(P_1,\ldots,P_{d-2},\Delta_i,\Delta_j) \ \ \text{for all $i,j$.}
\]

In the general case, 
we show in Lemma \ref{lem:AlgebraicEquivalence} that $Z_{i} \cap D_{i+1}$ is a union of irreducible components that are algebraically equivalent to each other.
We may therefore define $Z_{i+1}$ to be any one of these irreducible components. 
Since algebraic equivalence implies homological equivalence \cite[Chapter 19]{fulton}, we have
\[
\int_S H_iH_j=\int_Z H_iH_j=\lambda^{-1} \, \MV(P_1,\ldots,P_{d-2},\Delta_i,\Delta_j) \ \ \text{for all $i,j,$}
\]
where $\lambda$ is a positive rational number.
\end{proof}

Since convex bodies $A_1,\ldots,A_{d-2} \subseteq \mathbb{R}^d$ can be approximated by rational polytopes,  Proposition~\ref{cor:oneposeigen} implies the following classical result in convex geometry.\footnote{The same argument works more generally when $\Delta_i$'s are arbitrary convex bodies. This is one of the main conclusions of Brunn--Minkowski theory, equivalent to the Alexandrov--Fenchel inequality. See Section \ref{sec4} for a more complete discussion. We refer to  \cite[Section 7.3]{Schneider} for Alexandrov's original proof.} 

\begin{corollary}\label{cor:projectionvolumes}
For any convex bodies $A_1,\ldots,A_{d-2} \subseteq \mathbb{R}^d$, the $n \times n$ symmetric matrix with entries $\MV(A_1,\ldots,A_{d-2},\Delta_i,\Delta_j)$  is Lorentzian.
\end{corollary}

By \cite[Theorem 5.3.1]{Schneider}, 
when 
$m_i=1$ for all $i$, we have
\[
 \MV(A_1,\ldots,A_{d-2},\Delta_i,\Delta_j)=
\MV(\pi_{ij}(A_1),\ldots,\pi_{ij}(A_{d-2}))
\ \ \text{for any $i \neq j$}.
\]
Since the mixed volume is zero when $i=j$,
 condition ($1$) implies condition ($2$) in Conjecture~\ref{con:realizabilitybypolytopes}.


The remainder of this section is devoted to a general lemma used in the proof of Theorem~\ref{lem:polytopesgivesurface}.

\begin{lemma}\label{lem:AlgebraicEquivalence}
Let $f: Y\to X$ is a proper surjective morphism between irreducible algebraic varieties over an algebraically closed field. Then the irreducible components of a general fiber of $f$ are algebraiclly equivalent to each other in $Y$.
\end{lemma}

The same statement holds for the scheme-theoretic irreducible components of a general fiber, that is, the scheme-theoretic closures of the complements of the rest of the irreducible components within the fiber.

\begin{proof}
First, we consider the case when $f_* \mathcal{O}_Y=\mathcal{O}_X$ and $Y$ is normal.  
    By \cite{Fujita,FujitaCorrections}, the function field $K(X)$ is algebraically closed in the function field $K(Y)$, and hence $K(Y)$ is geometrically irreducible over $K(X)$ by \cite[Lemma~10.47.8]{StackProject}. 
    It follows that the generic fiber of $f$ is geometrically irreducible , and hence a general fiber of $f$ is irreducible \cite[Lemma~37.27.5]{StackProject}.


    Second, we consider the case when $f$ is finite. In this case, the assertion follows from the fact that any two points in $Y$ are algebraically equivalent in $Y$. 



We now consider the general case. 
Let $g: Y'\to Y$ be the normalization.
    Consider the composition $Y'\xrightarrow{g}Y\xrightarrow{f} X$, and denote its Stein factorization by $Y' \xrightarrow{f'}X'\xrightarrow{g'}X$. We have the commutative diagram
    \[
    \xymatrix{
    Y'\ar[d]_{g}\ar[r]^{f'}&X'\ar[d]^{g'}\\
    Y\ar[r]^{f}&X
    }
    \]
    where $g': X'\to X$ is a finite morphism and $f': Y'\to X'$ is a proper morphism satisfying $f'_{*}(\mathcal{O}_{Y'})=\mathcal{O}_{X'}$. 
Consider the closed subset
\[
V'=\{x'\in X' \,|\, \text{$f'^{-1}(x')$ is reducible or contained in the exceptional locus of $g$}\}.
\]
The analysis of the first case applies to $f'$, and we know that the image of $V'$ in $X$ is a proper closed subset of $X$. 
    By the constructibility of the relative Samuel functions \cite[Theorem 4.15]{LJT}, we know that the relative Hilbert--Samuel multiplicity is constructible.
It follows that two general fibers of $f'$ have the same \emph{geometric multiplicity}, the length of the local ring at the generic point.\footnote{Here we use the assumption that the base field is perfect, since otherwise the relative Hilbert--Samuel multiplicity is not constructibe \cite[Section 1]{Bennett}. For an example of a Stein factorization where the first factor has a non-reduced general fiber, see \cite{FujitaCorrections}.} 

Let $x$ be a general point of $X$ so that, in particular, $x$ is not in the image of $V'$ and the fiber of $g'$ over $x$ is supported on the smooth locus of $X'$.
Let $F_1,\ldots,F_k$ be the irreducible components of the fiber of $g' \circ f'$ over $x$. Since a general fiber of $g'$ consists of points of the same multiplicity and general fibers of $f'$ are irreducible and of the same multiplicity,
 the reduced schemes $F_i$ are algebraically equivalent to each other in $Y'$.
Since $g$ is birational, this shows that the irreducible components of $f^{-1}(x)$  are algebraically equivalent to each other in $Y$.
\end{proof}

\section{Projections and mixed volumes of four-dimensional bodies}\label{sec:polytopes}

To each pair of convex bodies $A$ and $B$ in $\RR^4$, we associate a vector of real numbers 
\[
A\wedge B=(p_{12},p_{13},p_{14},p_{23},p_{24},p_{34}) \in \medwedge^2\, \RR^4,
\]
where the components $p_{ij}$ are defined as the mixed volumes of the projections of $A$ and $B$. 
With this notation, the basic properties of mixed volumes read
\[
A \wedge B=B\wedge A  \ \ \text{and} \ \  \quad (\lambda_1 A_1+\lambda_2 A_2)\wedge B=\lambda_1( A_1 \wedge B )+\lambda_2(A_2\wedge B) \ \ \text{for $\lambda_1,\lambda_2 \ge 0$.}
\]
The main contributions of this section are Theorems~\ref{thm:mainreal} and ~\ref{thm:mainQpoly}, which characterize real and rational vectors of the form $A \wedge B$.
We also provide a proof of Theorem~\ref{thm:convexbody} when $\mathbf{p}$ is in $\Delta(\mathbb{T}_2)^\circ$. The proof of Theorem~\ref{thm:convexbody} when $\mathbf{p}$ is in $\partial\Delta(\mathbb{T}_2)$ will be given in the next section.

These results have direct applications to realizability questions concerning $(\PP^1)^4$, as we demonstrate in the following section.
%

The following reformulation of the Lorentzian condition for $4$-by-$4$ matrices from  \cite{BHKL} will play a critical role in the proofs of the above characterizations.

\begin{proposition}\label{prop:T2}
A point $\mathbf{p}=(p_{12},p_{13},p_{14},p_{23},p_{24},p_{34}) \in \medwedge^2\, \RR^4$ is in $\Delta(\mathbb{T}_2)$ if and only if
\[  
L(\mathbf{p})\coloneq\begin{pmatrix}
0&p_{12}&p_{13}&p_{14}\\
p_{12}&0&p_{23}&p_{24}\\
p_{13}&p_{23}&0&p_{34}\\
p_{14}&p_{24}&p_{34}&0
\end{pmatrix}
\ \ \text{is a Lorentzian matrix.}
\]
\end{proposition}

\begin{proof}
A direct computation reveals that 
  \begin{multline*}
      \det L(\mathbf{p})=-\left(\sqrt{p_{12}p_{34}}+\sqrt{p_{13}p_{24}}+\sqrt{p_{14}p_{23}}\right)\cdot(-\sqrt{p_{12}p_{34}}+\sqrt{p_{13}p_{24}}+\sqrt{p_{14}p_{23}})\\
      \cdot(\sqrt{p_{12}p_{34}}-\sqrt{p_{13}p_{24}}+\sqrt{p_{14}p_{23}})\cdot(\sqrt{p_{12}p_{34}}+\sqrt{p_{13}p_{24}}-\sqrt{p_{14}p_{23}}).
  \end{multline*}
  Note that any symmetric $3$-by-$3$ matrix with nonnegative off-diagonal and zero diagonal entries is Lorentzian. Therefore, by Cauchy's interlacing theorem,  $L(\mathbf{p})$ is Lorentzian if and only if the product 
\[
(-\sqrt{p_{12}p_{34}}+\sqrt{p_{13}p_{24}}+\sqrt{p_{14}p_{23}})(\sqrt{p_{12}p_{34}}-\sqrt{p_{13}p_{24}}+\sqrt{p_{14}p_{23}})(\sqrt{p_{12}p_{34}}+\sqrt{p_{13}p_{24}}-\sqrt{p_{14}p_{23}})
\]
  is nonnegative. 
  Since the sum of every two factors is nonnegative, the product is nonnegative if and only if every factor is nonnegative.
  Such nonnegativity corresponds precisely to the triangle condition defining $\Delta(\mathbb{T}_2)$. 
\end{proof}

For the remainder of this section, we write $\mathbf{k}$ for $\QQ$ or $\RR$.

\begin{definition}
We define an action
of the multiplicative group $\mathbf{k}_{>0}\times \mathbf{k}_{>0}^4$ on $\medwedge^2\, \mathbf{k}^{4}$ by
\[
(\lambda,c_1,c_2,c_3,c_4) \cdot (p_{12},p_{13},p_{14},p_{23},p_{24},p_{34})=\lambda (c_1c_2p_{12},c_1c_3p_{13},c_1c_4p_{14},c_2c_3p_{23},c_2c_4p_{24},c_3c_4p_{34}).
\]
We say that $\mathbf{p}$ and $\mathbf{q}$ are \emph{equivalent over $\mathbf{k}$} 
if $\mathbf{p}$ and $\mathbf{q}$ are in the same orbit of $\mathbf{k}_{>0}\times \mathbf{k}_{>0}^4 \rtimes S_4$.
\end{definition}

Note that the action of the diagonal $\mathbf{k}_{>0}$ is not subsummed by the action of $\mathbf{k}_{>0}^4$ if $\mathbf{k}=\QQ$.
Clearly, the action of $\mathbf{k}_{>0}\times \mathbf{k}_{>0}^4 \rtimes S_4$ preserves the sets $\Delta(\T2)$ and $\Delta^{\circ}(\T2)$. The next lemma shows that it also preserves the set of vectors of the form $A \wedge B$. 

\begin{lemma}
\label{lem:orbit}
    Let $A_1$, $A_2$, and $A$ be convex polytopes with vertices in $\mathbf{k}^4$, and let $\mathbf{p} \in \medwedge^2\, \mathbf{k}^4$.
    \begin{enumerate}[(1)]\itemsep 5pt
        \item If $\mathbf{p}$ is equivalent to $A_1\wedge A_2$ over $\mathbf{k}$, then $\mathbf{p}=B_1\wedge B_2$ for some convex polytopes $B_1$ and $B_2$ with vertices in $\mathbf{k}^4$. 
        \item If  $\mathbf{p}$ is equivalent to $A\wedge A$ over $\RR$, then $\mathbf{p}=B\wedge B$ for some convex polytope $B$ in $\mathbb{R}^4$.
    \end{enumerate}
Analogous statements hold for convex bodies $A_1$, $A_2$, and $A$ in $\RR^4$.
\end{lemma}

\begin{proof}
It is enough to consider the action of $\mathbf{k}_{>0}\times \mathbf{k}_{>0}^4$.
For the first statement, suppose that $\mathbf{p}=(\lambda,c_1,c_2,c_3,c_4)\cdot A_1 \wedge A_2$.
We take
\[
B_1=(c_1^{-1},c_2^{-1},c_3^{-1},c_4^{-1}) \cdot A_1
\ \ \text{and} \ \ 
B_2=\lambda c_1c_2c_3c_4 (c_1^{-1},c_2^{-1},c_3^{-1},c_4^{-1}) \cdot A_2.
\]
For the second statement, suppose that 
$\mathbf{p}=(\lambda,c_1,c_2,c_3,c_4)\cdot A\wedge A$. We take
\[
B=\lambda^{1/2}c_1^{1/2}c_2^{1/2}c_3^{1/2}c_4^{1/2}(c_1^{-1},c_2^{-1},c_3^{-1},c_4^{-1}) \cdot A.\qedhere
\]
\end{proof}

The following lemma characterizes the points in $\Delta(\mathbb{T}_2) \cap \medwedge^2\, \mathbf{k}^4$ that have at least one zero entry, up to equivalence over $\mathbf{k}$.

\begin{lemma}\label{lem:zero-entry}
If $\mathbf{p}$ is a nonzero vector in $\Delta(\mathbb{T}_2) \cap \medwedge^2\, \mathbf{k}^4$ with a zero entry, then $\mathbf{p}$ is equivalent over $\mathbf{k}$ to exactly one of the following six points: 
\begin{align*}
 (1,1,1,1,1,0), \ \ (1,1,0,1,0,0), \ \ (0,1,1,1,1,0),  \ \ (1,1,1,0,0,0),
  \ \ (1,1,0,0,0,0), \ \ (1,0,0,0,0,0).
\end{align*}
Each one of these points is of the form  $A \wedge A$ for some lattice polytope $A$.
\end{lemma}

\begin{proof}
We prove the first statement. If $\mathbf{p} \in \Delta(\mathbb{T}_2)$ has a zero entry, then, either all three products $p_{12}p_{34}$, $p_{13}p_{24}$, $p_{14}p_{23}$ are zero, or exactly one of them is zero. 
In the former case, it is straightforward to verify that $\mathbf{p}$ is equivalent over $\mathbf{k}$ to one of the following four points:
\[
(1,1,0,1,0,0), \quad (1,1,1,0,0,0), \quad (1,1,0,0,0,0), \quad (1,0,0,0,0,0).
\]
In the latter case, we may assume without loss of generality that $p_{34}=0$ and thus $p_{13}p_{24} = p_{14}p_{23}$ is nonzero. Then $\mathbf{p}$ is equivalent over $\mathbf{k}$ to exactly one of the points
\[
(0,1,1,1,1,0) \ \ \text{or}\ \  (1,1,1,1,1,0),
\]
depending on whether $p_{12}$ is zero or nonzero, respectively.
This proves the first statement.

The second statement is also a straightforward computation:
\begin{align*}
&A_2 \wedge A_2=(1,1,1,1,1,0) \ \ \text{when} \ \ A_2=\text{conv}(\mathbf{e}_1+\mathbf{e}_2,\mathbf{e}_3,\mathbf{e}_4), \\
&A_3 \wedge A_3=(1,1,0,1,0,0) \ \text{when} \ \ A_3=\text{conv}(\mathbf{0},\mathbf{e}_1+\mathbf{e}_2+\mathbf{e}_3,\mathbf{e}_4), \\
&A_4 \wedge A_4=(0,1,1,1,1,0) \ \ \text{when} \ \ A_4=\text{conv}(\mathbf{0},\mathbf{e}_1+\mathbf{e}_2,\mathbf{e}_3+\mathbf{e}_4), \\
   &A_5 \wedge A_5=(1,1,1,0,0,0) \ \ \text{when} \ \ A_5=\text{conv}(\mathbf{e}_2,\mathbf{e}_3,\mathbf{e}_4),\\
   &A_6 \wedge A_6=(1,1,0,0,0,0) \ \ \text{when} \ \ A_6=\text{conv}(\mathbf{0},\mathbf{e}_2+\mathbf{e}_3,\mathbf{e}_4), \\
  &A_7 \wedge A_7=(1,0,0,0,0,0) \ \ \text{when} \ \ A_7=\text{conv}(\mathbf{0},\mathbf{e}_3,\mathbf{e}_4). \qedhere
\end{align*}
\end{proof}


We now proceed to prove Theorem~\ref{thm:convexbody}.
The first statement,  the necessity of the condition $\mathbf{p} \in \Delta(\mathbb{T}_2)$, 
directly follows from Propositions~\ref{prop:T2} and 
Corollary~\ref{cor:projectionvolumes}.
For the second statement, the sufficiency of the condition $\mathbf{p} \in \Delta(\mathbb{T}_2)^\circ$, we construct a convex polytope $A$ in $\RR^4$ satisfying $\mathbf{p}=A\wedge A$.
The construction relies on the following symmetrization lemma.

\begin{lemma}\label{lem:symmetrization}
For any vector of positive real numbers $\mathbf{p}=(p_{12},p_{13},p_{14},p_{23},p_{24},p_{34})$,
there is a point $\mathbf{q}=(q_{12},q_{13},q_{14},q_{14},q_{13},q_{12})$ equivalent to $\mathbf{p}$ over $\RR$.
\end{lemma}

\begin{proof}
Let $\lambda$ be the reciprocal of $\sqrt{p_{12}p_{13}p_{14}p_{23}p_{24}p_{34}}$, and take
\[
c_1=\sqrt{p_{23}p_{24}p_{34}}, \ \ c_2=\sqrt{p_{13}p_{14}p_{34}}, \ \  
c_3=\sqrt{p_{12}p_{14}p_{24}}, \ \ c_4=\sqrt{p_{12}p_{13}p_{23}}.
\]
It is straightforward to check that $(\lambda, c_1,c_2,c_3,c_4) \cdot \mathbf{p}$ is equal to 
\[
 (\sqrt{p_{12}p_{34}}, \sqrt{p_{13}p_{24}}, \sqrt{p_{14}p_{23}}, \sqrt{p_{14}p_{23}}, \sqrt{p_{13}p_{24}}, \sqrt{p_{12}p_{34}}). \qedhere
\]
\end{proof}

There are exactly three ways to partition a set of four elements into two subsets of two elements each. The corresponding line segments in $\RR^4$ will be the building blocks of $A$:
\[
D_{12|34}\coloneq \text{conv}(\mathbf{e}_1+\mathbf{e}_2,\mathbf{e}_3+\mathbf{e}_4), \ 
D_{13|24}\coloneq \text{conv}(\mathbf{e}_1+\mathbf{e}_3,\mathbf{e}_2+\mathbf{e}_4), \ 
D_{14|23}\coloneq \text{conv}(\mathbf{e}_1+\mathbf{e}_4,\mathbf{e}_2+\mathbf{e}_3).
\]
A straightforward computation shows:
\begin{align*}
D_{12|34}\wedge D_{13|24}&=(2,2,0,0,2,2), \qquad D_{12|34}\wedge D_{12|34}=(0,0,0,0,0,0),\\
D_{12|34}\wedge D_{14|23}&=(2,0,2,2,0,2), \qquad D_{13|24}\wedge D_{13|24}=(0,0,0,0,0,0),\\
D_{13|24}\wedge D_{14|23}&=(0,2,2,2,2,0), \qquad D_{14|23}\wedge D_{14|23}=(0,0,0,0,0,0).
\end{align*}

\begin{proof}[Proof of Theorem~\ref{thm:convexbody} when $\mathbf{p}$ is in  $\Delta(\mathbb{T}_2)^\circ$.]
We first construct a convex polytope with projection areas given by $\mathbf{p}$. 
By Lemmas~\ref{lem:orbit} and ~\ref{lem:symmetrization}, we may assume that the given point $\mathbf{p}$ in $\Delta(\mathbb{T}_2)^\circ$ is of the form $(p_{12},p_{13},p_{14},p_{14},p_{13},p_{12})$.
By permuting the coordinates, we may suppose that $p_{12}\geq p_{13}\geq p_{14}$, which gives
\[
a\coloneq p_{12}-p_{13} \ge 0, \quad b \coloneq p_{13}-p_{14} \ge 0, \quad c \coloneq p_{13}+p_{14}-p_{12}>0.
\]
Note that $c$ is positive because  $\mathbf{p}$ is in the interior of $\Delta(\mathbb{T}_2)$. From the computation above, 
\[
\big(xD_{12|34}+yD_{13|24}+zD_{14|23}\big) \wedge  \big(xD_{12|34}+yD_{13|24}+zD_{14|23}\big)=\big(p_{12},p_{13},p_{14},p_{14},p_{13},p_{12}\big)
\]
for indeterminates $x$, $y$, $z$, translates into the system of equations
\[
xy=u\coloneq (a+b+d)/4, \quad xz=v\coloneq (a+d)/4, \quad yz=w\coloneq d/4,\quad \text{where $d=c/2>0$.}
\]
This system has a unique positive solution
\[
x=\sqrt{uv/w}, \quad y=\sqrt{uw/v}, \quad z=\sqrt{vw/u}. 
\]
This proves that there is a convex polytope $A$ such that $A \wedge A=\mathbf{p}$.

We now show that there is a smooth convex body $A$ with $A \wedge A=\mathbf{p}$.
Choose smooth families of smooth convex bodies $D_{12|34}(\epsilon),D_{13|24}(\epsilon),D_{14|23}(\epsilon)$ that converge to the line segments $D_{12|34},D_{13|24},D_{14|23}$ as $\epsilon$ goes to zero, and consider the Minkowski sum 
\[
A =xD_{12|34}(\epsilon)+yD_{13|24}(\epsilon)+zD_{14|23}(\epsilon).
\]
With $u$, $v$, $w$ defined as above, the system of equations $A \wedge A=\mathbf{p}$ translates to
\[
xy+f(x,y,z,\epsilon)=u, \quad xz+g(x,y,z,\epsilon)=v, \quad yz+h(x,y,z,\epsilon)=w,
\]
where $f$, $g$, $h$ are smooth functions satisfying 
\[
f(x,y,z,0)=g(x,y,z,0)=h(x,y,z,0)=0 \ \  \text{for all $x,y,z>0$}.
\]
We consider the function $F:\RR^4_{\ge 0}\to \RR^3$ given by
\[
F(x,y,z,\epsilon ) \coloneq (xy+f(x,y,z,\epsilon), xz+g(x,y,z,\epsilon),yz+h(x,y,z,\epsilon)).
\]
By modifying the families  $D_{12|34}(\epsilon),D_{13|24}(\epsilon),D_{14|23}(\epsilon)$ if necessary, we may suppose that $F$ extends to a smooth function on a neighborhood of the point 
\[
(\sqrt{uv/w}, \sqrt{uw/v}, \sqrt{vw/u},0) \in F^{-1}(u,v,w).
\]
For example, we may take $D_{12|34}(\epsilon)$ to be the ellipsoid of maximal volume inscribed in the $\epsilon$-neighborhood of $D_{12|34}$.\footnote{Let $C^\infty_+$ be the class of convex bodies with $C^\infty$ boundary and nonvanishing Gaussian curvature. Since $C^\infty_+$ is closed under the Minkowski sum, the resulting convex body $A$ will be in the class $C^\infty_+$.}
The leading maximal minor of the Jacobian of $F$ at the point 
is $-2\sqrt{uvw}$, which is nonzero. 
Thus, by the implicit function theorem,
 there is a germ of a smooth curve $(x(\epsilon),y(\epsilon),z(\epsilon))$ passing through the point that satisfies, for all sufficiently small $\epsilon$, 
\[
F(x(\epsilon),y(\epsilon),z(\epsilon),\epsilon)=(u,v,w).
\]
This shows that there are positive numbers $x,y,z$ and $\epsilon$ that satisfy $A \wedge A=\mathbf{p}$.
\end{proof}


We now characterize vectors of mixed areas of coordinate projections of pairs of convex bodies in $\RR^4$:

\begin{theorem}\label{thm:mainreal}
We have the equalities
\[
\Delta(\mathbb{T}_2)=\left\{A \wedge B \;\middle|\; \text{$A,B$ are convex bodies in $\RR^4$}\right\}
=\left\{A \wedge B \;\middle|\; \text{$A,B$ are convex polytopes in $\RR^4$}\right\}.
\]
\end{theorem}

\begin{proof}
As previously noted, the necessity of the condition $\mathbf{p} \in \Delta(\mathbb{T}_2)$ follows from Proposition~\ref{prop:T2} and 
Corollary~\ref{cor:projectionvolumes}.
We prove the sufficiency. 
By Theorem~\ref{thm:convexbody},  we may suppose that $\mathbf{p} \in \partial \Delta(\mathbb{T}_2)$.
By Lemma~\ref{lem:zero-entry}, we may suppose that every entry of $\mathbf{p}$ is positive.
By Lemma~\ref{lem:symmetrization}, we may suppose that $\mathbf{p}=(p_{12},p_{13},p_{14},p_{14},p_{13},p_{12})$. 
By permuting coordinates, we may suppose that $p_{12}=p_{13}+p_{14}$. Under these assumptions, we find that
\[
\big(p_{13}D_{13|24}+p_{14}D_{14|23}\big)\wedge \big( 2^{-1}D_{12|34}\big)=\big(p_{12},p_{13},p_{14},p_{14},p_{13},p_{12}\big). \qedhere
\]
\end{proof}



We now characterize vectors of mixed areas of coordinate projections of pairs of rational polytopes in $\RR^4$:

\begin{theorem}\label{thm:mainQpoly}
We have the equality
\[
\Delta(\mathbb{T}_2) \cap \medwedge^2\, \QQ^4=\left\{A \wedge B\;\middle|\; \text{$A,B$ are rational polytopes in $\RR^4$}\right\}.
\]
\end{theorem}

The main technical challenge is the absence of Lemma~\ref{lem:symmetrization} for $\mathbf{k}=\QQ$.
We begin by analyzing the symmetric case.

\begin{lemma}\label{lem:symmetrcoverQ}
If $\mathbf{p}=(p_{12},p_{13},p_{14},p_{14},p_{13},p_{12}) \in \Delta(\mathbb{T}_2)$ is a vector of nonnegative rational numbers 
satisfying $p_{12} \ge p_{13} \ge p_{14}$, then there is a rational polytope $A$ such that
$\mathbf{p}=A\wedge D_{12|34}$.
\end{lemma}

\begin{proof}
We consider the lattice polytopes
\[
P \coloneq D_{13|24}+D_{14|23}, \quad
Q \coloneq D_{13|24}, \quad
R\coloneq \text{conv}(\mathbf{e}_1,\mathbf{e}_2)+\text{conv}(\mathbf{e}_3,\mathbf{e}_4). 
\]
These polytopes give the following projection mixed volumes
\[
P\wedge D_{12|34} =(4,2,2,2,2,4),\quad
Q\wedge D_{12|34} =(2,2,0,0,2,2), \quad
R\wedge D_{12|34} =(2,2,2,2,2,2).
\]
Therefore, $x\coloneq (p_{12}-p_{13})/2 \ge 0$, $y \coloneq (p_{13}-p_{14})/2 \ge 0$, $z \coloneq (p_{13}+p_{14}-p_{12})/2 \ge 0$ satisfy 
\[
\big(xP+yQ+zR\big)\wedge D_{12|34}=\big(p_{12},p_{13},p_{14},p_{14},p_{13},p_{12}\big). \qedhere
\]
\end{proof}

We continue with an analysis of the boundary case.

\begin{lemma}\label{lem:Qdegenerate but no zero}
If $\mathbf{p}=(p_{12},p_{13},p_{14},p_{23},p_{24},p_{34}) \in \partial\Delta(\mathbb{T}_2)$ is a vector of positive rational numbers,
then 
there is a  vector
$\mathbf{q}=(q_{12},q_{13},q_{14},q_{14},q_{13},q_{12}) \in \partial\Delta(\mathbb{T}_2)$ of positive rational numbers  equivalent to 
$\mathbf{p}$ over $\QQ$.
\end{lemma}

\begin{proof}
Without loss of generality, we may assume $p_{14}=p_{23}=1$ and
\[
\sqrt{p_{12}p_{34}}+1=\sqrt{p_{13}p_{24}}.
\]
Taking the square, we see that $\sqrt{p_{12}p_{34}}$ is rational, and hence, so is $\sqrt{p_{13}p_{24}}$. 
Take 
\[
c_1=\sqrt{p_{12}p_{34}}\sqrt{p_{13}p_{24}}, \ \  c_2=p_{13}\sqrt{p_{12}p_{34}}, \ \  
c_3=p_{12}\sqrt{p_{13}p_{24}}, \ \  c_4=p_{12}p_{13}.
\]
It is straightforward to check that $(c_1^{-1}c_4^{-1},c_1,c_2,c_3,c_4) \cdot \mathbf{p}$ is equal to 
\[
(\sqrt{p_{12}p_{34}}, \sqrt{p_{13}p_{24}}, 1, 1, \sqrt{p_{13}p_{24}}, \sqrt{p_{12}p_{34}}). \qedhere
\]
\end{proof}




We conclude this section with the proof of Theorem \ref{thm:mainQpoly}.

\begin{proof}[Proof of Theorem \ref{thm:mainQpoly}]

 If $\mathbf{p}$ has at least one zero entry, the existence follows from Lemma~\ref{lem:zero-entry}. 
 If $\mathbf{p}$ has only positive entries and satisfies a degenerate triangle inequality, the existence follows from Lemmas~\ref{lem:symmetrcoverQ} and ~\ref{lem:Qdegenerate but no zero}.

 Suppose $\mathbf{p}$ has only positive entries and satisfies the nondegenerate triangle inequalities.
 By Lemma~\ref{lem:symmetrization}, we have $\mathbf{p}$ is equivalent to $\mathbf{q}=(q_{12}, q_{13},q_{14}, q_{14}, q_{13}, q_{12})$ over $\RR$, where
\[
q_{12}=\sqrt{p_{12}p_{34}}, \quad q_{13}=\sqrt{p_{13}p_{24}}, \quad q_{14} =\sqrt{p_{14}p_{23}}.
\]
Without loss of generality, we may suppose that $q_{12} \ge q_{13} \ge q_{14}$.

 Choose rational polytopes $R_{12},R_{13},R_{14},R_{23},R_{24},R_{34}$ in $\RR^4$ such that
\begin{align*}
R_{12}\wedge D_{12|34} &=(0,1,1,1,1,1), \qquad R_{23}\wedge D_{12|34}=(1,1,1,0,1,1), \\
R_{13}\wedge D_{12|34} &=(1,0,1,1,1,1), \qquad R_{24}\wedge D_{12|34}=(1,1,1,1,0,1), \\
R_{14}\wedge D_{12|34} &=(1,1,0,1,1,1), \qquad R_{34}\wedge D_{12|34}=(1,1,1,1,1,0).
\end{align*}
For example, we may choose the convex polytopes
\begin{align*}
 R_{12} &=\text{conv}(\mathbf{e}_1+\mathbf{e}_2, \mathbf{e}_3+\mathbf{e}_4, \mathbf{e}_1, \mathbf{0}),
 \ \ \quad
 R_{23} =\text{conv}(\mathbf{e}_1+\mathbf{e}_2, \mathbf{e}_3+\mathbf{e}_4, \mathbf{e}_3,\mathbf{e}_4),
 \\
 R_{13} &=\text{conv}(\mathbf{e}_1+\mathbf{e}_2, \mathbf{e}_3+\mathbf{e}_4, \mathbf{e}_2,\mathbf{e}_4), \quad
 R_{24} =\text{conv}(\mathbf{e}_1+\mathbf{e}_2,\mathbf{e}_3+\mathbf{e}_4, \mathbf{e}_1,\mathbf{e}_3), \\
 R_{14} &=\text{conv}(\mathbf{e}_1+\mathbf{e}_2, \mathbf{e}_3+\mathbf{e}_4, \mathbf{e}_2,\mathbf{e}_3),
   \quad 
R_{34} =\text{conv}(\mathbf{e}_1+\mathbf{e}_2, \mathbf{e}_3+\mathbf{e}_4,  \mathbf{e}_3, \mathbf{0}).
\end{align*}
For any point $\mathbf{v}\in \medwedge^2\, \RR^4$, we define a cone  centered at $\mathbf{v}$ by
\[
\mathcal{C}(\mathbf{v}) = \left\{\mathbf{v}+\sum_{ij}c_{ij} R_{ij}\wedge D_{12|34}\;\middle|\; c_{ij}\ge 0\right\},
\]
and denote its interior by $\mathcal{C}^\circ(\mathbf{v})$.
We write $S$ for the subspace of $\medwedge^2\, \RR^4$ consisting of points 
\[
(a_1,a_2,a_3,a_3,a_2,a_1) \ \ \text{for} \ \ a_1,a_2,a_3 \in \RR.
\]

Take $\mathbf{q}_{-} =\mathbf{q}-\varepsilon (1,1,1,1,1,1)$, where $\varepsilon>0$ is chosen small enough so that $\mathbf{q}_{-}$ has positive entries and satisfies the nondegenerate triangle inequalities. Clearly, $\mathbf{q}$ is in  $\mathcal{C}^\circ(\mathbf{q}_{-})$.

Since $\mathcal{C}(\mathbf{q}_-)$ is a full-dimensional cone whose interior intersects $S$, there is a rational point arbitrarily close to $\mathbf{q}_{-}$ in the intersection $\mathcal{C}^\circ(\mathbf{q}_{-})\cap S$. 
Thus, there is a rational point 
\[
\mathbf{r}=(r_{12}, r_{13}, r_{14}, r_{14}, r_{13}, r_{12}) \ \ \text{satisfying} \ \  r_{12} \ge r_{13} \ge r_{14} \ \  \text{and} \ \   \mathbf{q} \in \mathcal{C}^\circ_{\mathbf{r}}.
\]
By Lemma~\ref{lem:symmetrcoverQ}, we have $\mathbf{r}=A \wedge D_{12|34}$ for some rational polytope $A$. 
Thus, any rational point in $\mathcal{C}_\mathbf{r}$ is of the form
\[
\big(A+c_{12} R_{12}+c_{13}R_{13}+c_{14}R_{14}+c_{23}R_{23}+c_{24}R_{24}+c_{34}R_{34}\big)\wedge  D_{12|34}
\]
for some nonnegative rational numbers $c_{ij}$.
Since
 $\QQ^4_{>0}$ is dense in  $\RR_{>0}^4$, there is a point  arbitrarily close to $\mathbf{q}$ that is equivalent to $\mathbf{p}$ over $\QQ$.
Clearly, we may take this rational point to be in $\mathcal{C}^\circ_{\mathbf{r}}$.
 We conclude by Lemma \ref{lem:orbit}.
\end{proof}

\section{Realizing the boundary of $\Delta(\mathbb{T}_2)$}
\label{sec4}
In this section, we finish the proof of Theorem~\ref{thm:convexbody}. We construct a family of convex bodies $A \subseteq \RR^4$ such that $A \wedge A$ sweeps out $\partial\Delta(\mathbb{T}_2) \simeq S^4$ and show that any such $A$ necessarily has a singular point in its boundary.
By Lemma~\ref{lem:zero-entry}, we may suppose that the given vector in $ \partial\Delta(\mathbb{T}_2)$ has all entries positive.
By Lemmas~\ref{lem:orbit} and ~\ref{lem:symmetrization}, we may suppose that the given vector is of the form $(s+1,s,1,1,s,s+1)$ for some $s \ge 1$. We will give an explicit description of the vertices of a convex polytope $A$ as a function of $s$.

Finding a convex polytope $A$ satisfying $A \wedge A=(s+1,s,1,1,s,s+1)$ is nontrivial. Nevertheless, once the vertices of $A$ are correctly guessed as a function of $s$, verifying that $A$ has the required property is straightforward, as demonstrated below.
 In the remainder of this section, we show how recent breakthroughs on the equality conditions of the Alexandrov--Fenchel inequality in \cite{MinkowskiEquality,AFEquality}  lead to a specific set of vertices of $A$.

\begin{proposition}\label{lem: T2boundry}
For any $s \ge 1$, there is a convex polytope $A \subseteq \RR^4$ such that 
\[
A \wedge A =(s+1,s,1,1,s,s+1).
\]
\end{proposition}

\begin{proof}
Take $c=\sqrt{1-1/s}$, and
consider the $4 \times 16$ matrices
\[
\begin{array}{r@{\;}c@{\;}l}
L &\coloneq& 
\left[
\begin{array}{*{20}r}
 \phantom{-}0 & -1 &  0 & -1 & -1 &  0 &  0 &  \phantom{-}0 &  \phantom{-}1 &  \phantom{-}1 &  1 &  0 &  0 & -1 &  1 &  0 \\
 0 & -1 &  0 &  1 &  1 &  0 &  0 &  0 &  1 &  1 & -1 &  0 &  0 & -1 & -1 &  0 \\
 1 &  0 & -1 &  0 &  0 & -1 & -1 &  1 &  0 &  0 &  0 &  1 & -1 &  0 &  0 &  1 \\
 1 &  0 & -1 &  0 &  0 &  1 &  1 &  1 &  0 &  0 &  0 & -1 & -1 &  0 &  0 & -1
\end{array}
\right], 
\\[3em]
M &\coloneq& 
\left[
\begin{array}{*{20}r}
  0 &  \phantom{-} 0 &  \phantom{-}0 &  0 &  1 &  0 & -1 & -1 & -1 &  0 &  0 &  1 &  1 &  1 & -1 &  0 \\
 -2 &  0 &  0 & -2 & -1 &  0 & -1 & -1 & -1 & -2 &  0 & -1 & -1 & -1 & -1 & -2 \\
  0 &  0 &  0 &  0 & -1 &  0 &  1 & -1 & -1 &  0 &  0 & -1 &  1 &  1 &  1 &  0 \\
 -2 &  0 &  0 &  0 & -1 & -2 & -1 & -1 & -1 & -2 & -2 & -1 & -1 & -1 & -1 &  0
\end{array}
\right].
\end{array}
\]
Let $A$ be the convex hull of the $16$ columns of $L+cM$ in $\RR^4$.
With patience, one can check that
\[
   |\pi_{12}(A)|=|\pi_{34}(A)|=8-4c^2, \quad |\pi_{13}(A)|=|\pi_{24}(A)|=4,\quad |\pi_{14}(A)|=|\pi_{23}(A)|=4-4c^2,
\]
where $|\pi_{ij}(A)|$ is the normalized volume in  $\RR^2$.
We have
\[
A \wedge A \sim \Bigg(\frac{2-c^2}{1-c^2},\frac{1}{1-c^2},1,1,\frac{1}{1-c^2},\frac{2-c^2}{1-c^2}\Bigg)=(s+1,s,1,1,s,s+1). \qedhere
\]
\end{proof}

We now explain how we obtained the vertices of $A$. 
Let $\Delta_i$ be the line segment joining $\mathbf{0}$ and the standard basis vector $\mathbf{e}_i$.
The main observation here is that $A \wedge A =(s+1,s,1,1,s,s+1)$ corresponds to the equality case 
\[
\MV(A,A,\Delta_1+\Delta_2,\Delta_3+\Delta_4)^2=\MV(A,A,\Delta_1+\Delta_2,\Delta_1+\Delta_2)
\MV(A,A,\Delta_3+\Delta_4,\Delta_3+\Delta_4),
\]
because both sides are equal to $4(s+1)^2$.
Unlike many classical geometric inequalities, which typically admit simple equality characterizations, the Alexandrov--Fenchel inequality has a rich family of equality cases that depends on the relative position of sigular boundary points on the convex bodies \cite[Section 7.6]{Schneider}. Two recent breakthroughs of Shenfeld--van Handel in \cite{MinkowskiEquality} and \cite{AFEquality}  settled respectively the special case of Minkowski inequality for arbitrary convex bodies and the general case of Alexandrov--Fenchel inequality for convex polytopes.

The statement we need is \cite[Theorem 1.3]{MinkowskiEquality}, which confirms Schneider's conjecture from \cite{SchneiderConjecture} in the setting of Minkowski's quadratic inequality.
    Let $P$, $Q$, and $A$ be convex bodies in $\RR^n$.
    A nonzero vector $\mathbf{u} \in \RR^n$ is called an \emph{$r$-extreme normal vector} of $A$ if there do not exist linearly independent normal vectors $\mathbf{u}_1,\ldots,\mathbf{u}_{r+2}$ at a boundary point of $A$ such that $\mathbf{u}=\mathbf{u}_1+\cdots+\mathbf{u}_{r+2}$.

\begin{theorem}[Shenfeld--van Handel]\label{thm:equalityAF}
Suppose that $A$ is $n$-dimensional and $\MV(A,\dots,A,Q,Q)$ is positive. In this case, 
\begin{equation*}
    \MV(A,\dots,A,P,Q)^2=\MV(A,\dots,A,P,P)\, \MV(A,\dots,A,Q,Q)
\end{equation*}
holds if and only if  there are $\lambda \ge 0$ and $\mathbf{t}\in\mathbb R^n$ such that $P$ and $\lambda Q+\mathbf{t}$ have the same supporting hyperplanes in all $1$-extreme normal directions of $A$.
\end{theorem}

When $A$ is a polytope containing the origin in its interior, $\mathbf{u}$ is a $1$-extreme normal direction if and only if there exists $c>0$ such that $c \mathbf{u} $ lies on the $1$-skeleton of  the polar dual of $A$.
The \emph{support function} of a convex body $P\subseteq \mathbb{R}^n$ is the convex function 
    \[
    h_P(\mathbf{u})\coloneq \max_{\mathbf{x}\in P} \langle \mathbf{x},\mathbf{u}\rangle.
    \]
    In terms of the support functions, the equality condition in Theorem~\ref{thm:equalityAF} is that, for some $\lambda \ge 0$ and  $\mathbf{t} \in \RR^n$, we have 
\[
h_P(\mathbf{u})=\lambda h_Q(\mathbf{u})+\langle \mathbf{t},\mathbf{u}\rangle \ \  \text{for every $1$-extreme normal direction $\mathbf{u}$ of $A$.}  
\]
Note that $\lambda$ is necessarily equal to the ratio  $\MV(A,\ldots,A,P,Q)/\MV(A,\ldots,A,Q,Q)$ because 
$$
\MV(A,\ldots,A,P,Q) = \MV(A,\ldots,A,\lambda Q+\mathbf{t},Q)=\lambda \MV(A,\ldots,A,Q,Q).$$
 When $P=\Delta_1+\Delta_2$ and $Q=\Delta_3+\Delta_4$, Theorem~\ref{thm:equalityAF} says the following.

\begin{corollary}\label{cor:equalityOURS}
If $A$ is a $4$-dimensional convex body such that $A\wedge A =(s+1,s,1,1,s,s+1)$ for $s \ge 1$,
then there exists $\mathbf{t}\in \mathbb R^4$  such that
\begin{equation*}
    \max(u_1,0)+\max(u_2,0)-\max(u_3,0)-\max(u_4,0) = \langle \mathbf{t}, \mathbf{u}\rangle
\end{equation*}
holds for all 1-extreme directions $\mathbf{u}=(u_1,u_2,u_3,u_4)$ of A.    
\end{corollary}


Corollary~\ref{cor:equalityOURS} poses a strong constraint on the convex bodies corresponding to a point in $\partial \Delta(\mathbb{T}_2)$. 
When $A$ is smooth, every normal direction is $1$-extreme, so no smooth convex body $A$ satisfies $A \wedge A=(s+1,s,1,1,s,s+1)$ for $s \ge 1$. When $A$ is a convex polytope, we can use Corollary~\ref{cor:equalityOURS} to determine the facet directions of $A$. 
We illustrate this for $\mathbf{t}=(\tfrac{1}{2},\tfrac{1}{2},-\tfrac{1}{2},-\tfrac{1}{2})$.

\begin{lemma}\label{lem:tequal1/2}
    Let $B$ be a convex polytope in $\RR^4$ containing the origin in its interior. Suppose that all of its vertices and edges are contained in the hypersurface
    \begin{equation*}
   \mathscr{T}\coloneq \Big\{\mathbf{u} \in\mathbb R^4\setminus 0 \mid \max(u_1,0)+\max(u_2,0)-\max(u_3,0)-\max(u_4,0) = \tfrac{1}{2}(u_1+u_2-u_3-u_4)\Big\}.
\end{equation*}
Then $B$ has at most 16 vertices, each lying on a ray generated by a column of the matrix
\[
N\coloneq \left[
\begin{array}{*{20}r}
 1 & \phantom{-} 1 & \phantom{-} 1 & \phantom{-} 1 & \phantom{-} 0 & \phantom{-} 0 & \phantom{-} 0 & \phantom{-} 0 & -1 & -1 & -1 & -1 &  \phantom{-}0 & \phantom{-} 0 & \phantom{-} 0 & \phantom{-} 0\\
 0 &  0 &  0 &  0 &  1 &  1 &  1 &  1 &  0 &  0 &  0 &  0 & -1 & -1 & -1 & -1\\
 1 &  0 & -1 &  0 &  1 &  0 & -1 &  0 &  1 &  0 & -1 &  0 &  1 &  0 & -1 &  0\\
 0 &  1 &  0 & -1 &  0 &  1 &  0 & -1 &  0 &  1 &  0 & -1 &  0 &  1 &  0 & -1
\end{array}
\right].
\]
\end{lemma}

The proof is straightforward, given the piecewise linear structure of $\mathscr{T}$, shown below in Figure~\ref{fig:3d-surface-plot}. 
\begin{figure}[htbp]
  \centering
  \includegraphics[trim={2cm 2.2cm 2cm 1.6cm},clip, width=0.5\textwidth]{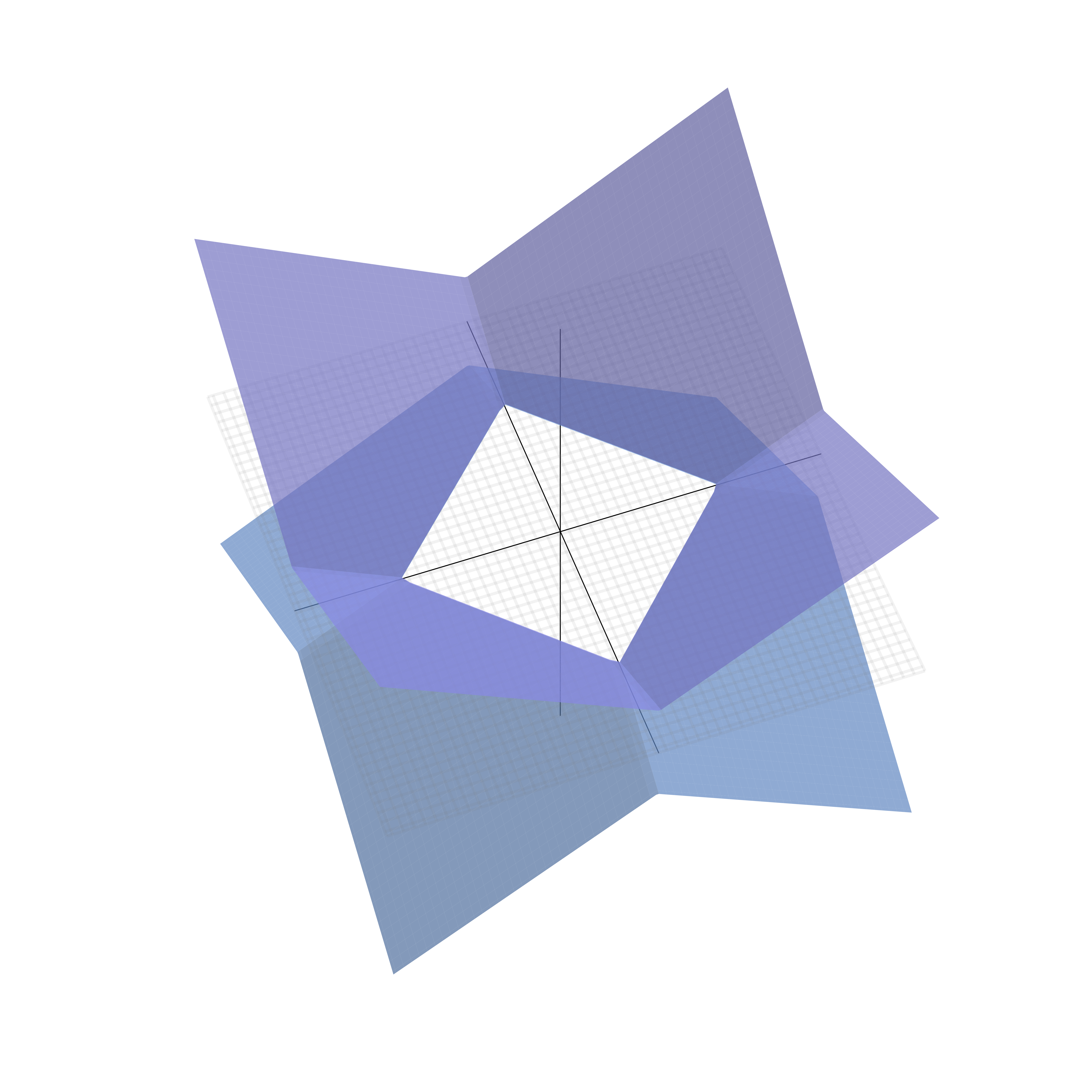}
  \caption{The hypersurface $\mathscr{T}$ in the affine chart $(u_1/u_4,u_2/u_4,u_3/u_4)$}
  \label{fig:3d-surface-plot}
\end{figure}
 Applying  
 Lemma~\ref{lem:tequal1/2} to the polar dual $B$ of $A$, we see that $A$ must be of the form
\[
    A=\bigcap_{i=1}^{16} \Big\{\mathbf{x} \in \mathbb R^4\mid  \langle \mathbf{x},  \mathbf{v}_i\rangle \geq a_i\Big\},
    \]
where $\mathbf{v}_i$ are the column vectors of the matrix $N$ and $a_1,\dots,a_{16}$ are some undetermined numbers.  
The area of the projection $\pi_{ij}(A)$ is a function of the 16 variables $a_i$. 
The combinatorial type of $A$ subdivides $\mathbb R^{16}$ into a finite number of polyhedral regions, and in each region, the area of $\pi_{ij}(A)$ is a quadratic polynomial in the variables $a_i$.

In principle, one can compute the subdivision of $\RR^{16}$ and the corresponding piecewise quadratic polynomial. For our purposes, it is enough to consider the following $1$-parameter family
 $A=A(c)$ given by the vector
\[
(a_1,\ldots,a_{16})=-(1,1,1,1,1,1+2c,1,1-2c,1,1+2c,1+2c,1,1,1,1-2c,1-2c), \ \ c \in [0,1].
\]
One can check that its vertices are given by the matrices $L$ and $M$ in  Proposition~\ref{lem: T2boundry}, and
\[
   |\pi_{12}(A)|=|\pi_{34}(A)|=8-4c^2, \quad |\pi_{13}(A)|=|\pi_{24}(A)|=4,\quad |\pi_{14}(A)|=|\pi_{23}(A)|=4-4c^2.
\]
    The method of finding $A$ does not depend on the specific choice of $\mathbf{t}$ or on the assumption that $A$ is a convex polytope. This provides a way to classify all convex bodies $A$ such that $A\wedge A=\mathbf{p}$ for a fixed $\mathbf{p}\in \partial \Delta(\mathbb{T}_2)$.
    



\section{Realizable classes in $(\PP^1)^4$}\label{Section4}

For $\mathbf{p}=(p_{12},p_{13},p_{14},p_{23},p_{24},p_{34}) \in \medwedge^2\, \QQ^4$, we write $\eta(\mathbf{p})$ for the homology class
\[
p_{12}[\PP^1 \times \PP^1 \times \PP^0 \times \PP^0]
+\cdots +p_{34}[\PP^0 \times \PP^0 \times \PP^1 \times \PP^1] \in \H_4((\PP^1)^4,\mathbb{Q}).
\] 
Theorem~\ref{thm:mainQ} states that $\eta(\mathbf{p})$
is realizable over $\QQ$ if and only if $\mathbf{p}\in \Delta(\mathbb{T}_2)$.

\begin{proof}[Proof of Theorem~\ref{thm:mainQ}]
We only need to collect the  pieces we already have.
The necessity of the condition $\mathbf{p} \in \Delta(\mathbb{T}_2)$ follows from Corollary~\ref{cor:oneposeigen} and Proposition~\ref{prop:T2}.
The sufficiency of the condition $\mathbf{p} \in \Delta(\mathbb{T}_2)$  follows from Theorem~\ref{lem:polytopesgivesurface} Theorem~\ref{thm:mainQpoly}.
\end{proof}

As noted before, $\mathbf{p}\in \Delta(\mathbb{T}_2) \cap \medwedge^2\, \mathbb{Z}^4$ does not imply that
$\eta(\mathbf{p})$ is realizable over $\mathbb{Z}$.

\begin{example}\label{exm:nonrepbut multiple is}
Consider the point $\mathbf{p}=(1,1,1,1,1,4) \in \partial\Delta(\mathbb{T}_2)$.
Proposition~\ref{cor:necconditions} below shows that $\eta=\eta(\mathbf{p})$ is not realizable over $\ZZ$. 
On the other hand, consider the lattice polytope
\[
P\coloneq \text{conv}(2\mathbf{e}_1,2\mathbf{e}_2,2\mathbf{e}_1+2\mathbf{e}_2,\mathbf{e}_3,\mathbf{e}_4,\mathbf{e}_3+\mathbf{e}_4) \subseteq \RR^4,
\]
which has projection areas $(2,2,2,2,2,8)$. Therefore,  by Theorem~\ref{lem:polytopesgivesurface}, $2\eta$ is realizable over $\ZZ$. 
In fact, by Theorem~\ref{thm:realization over Z}, the class $p \eta$ is realizable over $\ZZ$ for any prime number $p$.
\end{example}

In Theorem~\ref{thm:realization over Z}, we prove a more precise version of Theorem~\ref{thm:mainQ}:
If $\mathbf{p}$ is a primitive integral vector in $\Delta(\mathbb{T}_2)$ and $p$ is a prime number, then $\eta(\mathbf{p})$ is realizable over $\ZZ$ or $p\eta(\mathbf{p})$ is realizable over $\ZZ$.\footnote{An integral vector is said to be \emph{primitive} if its entries are coprime, that is, when the greatest common divisor of the entries is $1$.  In contrast to the other arguments in this paper, the proof of this statement requires the base field to have characteristic zero.} 
 


\begin{proposition}\label{cor:necconditions}
If $\eta(\mathbf{p})$ is realizable over $\ZZ$
and $p_{ij}> 0$, then $p_{kl}\leq p_{ik}p_{jl}+p_{il}p_{jk}$. 
\end{proposition}
 
Note that $\sqrt{p_{ij}p_{kl}}\leq \sqrt{p_{ik}p_{jl}}+\sqrt{p_{il}p_{jk}}$
implies  $p_{kl}\leq p_{ik}p_{jl}+p_{il}p_{jk}$ when  $p_{ij}>1$.
For $i=1$ and $j=2$, the points 
\[
(0,0,1,0,1,1), (0,0,0,1,1,1), (0,0,0,0,1,1), (0,0,0,0,0,1) \in \partial \Delta(\mathbb{T}_2)
\]
fail to satisfy both the assumption and the conclusion of Proposition~\ref{cor:necconditions}. 
The corresponding homology classes are realizable over $\mathbb{Z}$ by Lemma~\ref{lem:zero-entry} and Theorem~\ref{lem:polytopesgivesurface}. 
When $\mathbf{p} \in \Delta(\mathbb{T}_2)^\circ$, we know no other obstructions to the realizability of $\eta(\mathbf{p})$ over $\mathbb{Z}$ than the ones from Theorem~\ref{thm:mainQ} and Proposition~\ref{cor:necconditions}.

As before, for $E \subseteq [4]$, we write $\varphi_E$ for the coordinate projection $(\mathbb{P}^1)^{4} \to (\mathbb{P}^1)^{4-|E|}$ that forgets the coordinates labelled by $E$.
Following the usual abuse of notation, we use the same symbol $[V]$ for the Poincar\'e dual of the homology fundamental class $[V]$. 

\begin{proof}
Let $S$ be an irreducible surface in $(\PP^1)^4$ with  fundamental class $\eta(\mathbf{p})$. 
We show that 
$p_{34}\leq p_{13}p_{24}+p_{14}p_{23}$ when  $p_{12}> 0$. 
Since $p_{12}> 0$,
the coordinate projections 
$\varphi_3(S)$ and $\varphi_4(S)$ are irreducible hypersurfaces in $(\PP^1)^3$,  and their cohomology classes satisfy
\[
d_3[ \varphi_3(S)]=p_{24}H_1+p_{14}H_2+p_{12}H_4, \quad
d_4[ \varphi_{4}(S)]=p_{23}H_1+p_{13}H_2+p_{12}H_3,
\]
where $d_i$ is the degree of the map $S \to \varphi_i(S)$.
The inverse images
$\varphi_3^{-1}\varphi_3(S)$
and $\varphi_4^{-1}\varphi_4(S)$ are irreducible hypersurfaces in $(\PP^1)^4$, and they are necessarily distinct because $p_{12}>0$.
Thus $\varphi_3^{-1}\varphi_3(S)$
and $\varphi_4^{-1}\varphi_4(S)$ intersect generically transversally, and 
\begin{multline*}
d_3d_4[ \varphi_3^{-1}\varphi_3(S) \cap \varphi_4^{-1}\varphi_4(S)]=(p_{24}H_1+p_{14}H_2+p_{12}H_4)(p_{23}H_1+p_{13}H_2+p_{12}H_3)\\
=(p_{13}p_{24}+p_{14}p_{23})H_1H_2+p_{12}(p_{24}H_1H_3+p_{23}H_1H_4+p_{14}H_2H_3+p_{13}H_2H_4+p_{12}H_3H_4).
\end{multline*}
Since $S$ is an irreducible component of the intersection of  $\varphi_3^{-1}\varphi_3(S)$
and $\varphi_4^{-1}\varphi_4(S)$,
\[
\text{deg}\Big(\varphi_{12}:S \to (\PP^1)^2\Big) \le \text{deg}\Big(\varphi_{12}: \varphi_3^{-1}\varphi_3(S) \cap \varphi_4^{-1}\varphi_4(S)\to (\PP^1)^2\Big).
\]
Thus, by Poincar\'e duality, the above computation gives
\[
p_{34} \le d_3d_4p_{34} \le p_{13}p_{24}+p_{14}p_{23}.\qedhere
\]
\end{proof}

\begin{remark}
    A complete characterization of the classes in $\H_4(\PP^2\times \PP^2, \ZZ)$  realizable over $\ZZ$ is given in \cite{Huh13}. 
\end{remark}

Consider now the homology classes of surfaces in $(\PP^1)^4$ obtained by intersecting classes of hypersurfaces. We refer to these as \emph{complete intersection classes}.

\begin{definition}
Let $\Delta(\mathbb{T}_1)$ be the set of $(p_{12},p_{13},p_{14},p_{23},p_{24},p_{34}) \in \medwedge^2\,\RR^4$ such that
\[
p_{ij} \ge 0 \ \ \text{and} \ \ 
p_{ij}p_{kl}+ p_{ik}p_{jl} \ge p_{il}p_{jk} \ \  \text{for any $i,j,k,l$.}
\]
\end{definition}

We use $\Delta(\mathbb{T}_1)$ to characterize the complete intersection classes in $(\PP^1)^4$, up to a positive rational multiple:

\begin{theorem}\label{thm:CI}
For any nonzero $\mathbf{p} \in \medwedge^2\,\QQ^4$, some positive rational multiple of $\eta(\mathbf{p})$ is a complete intersection class if and only if $\mathbf{p} \in \Delta(\mathbb{T}_1)$.
\end{theorem}

    It is not true that $\eta(\mathbf{p})$ is alwyas a complete intersection class for any $\mathbf{p}\in \Delta(\mathbb{T}_1) \cap \medwedge^2\, \mathbb{Z}^4$. 
    For example,  $\eta(1,1,1,1,1,1)$ is not a complete intersection class, while $\eta(2,2,2,2,2,2)$ is a complete intersection class.

\begin{proof}
It will be useful to work with $\H^\bullet((\PP^1)^4,\QQ)$ instead of $\H_\bullet((\PP^1)^4,\QQ)$. We set
\[
 (q_{34},q_{24},q_{23},q_{14},q_{13},q_{12})\coloneq (p_{12},p_{13},p_{14},p_{23},p_{24},p_{34}),
\]
and, for any distinct indices $i,j,k,l$, define
\[
q_{ij|kl}\coloneq q_{ik}q_{jl}+q_{il}q_{jk}-q_{ij}q_{kl}=p_{ik}p_{jl}+p_{il}p_{jk}-p_{ij}p_{kl}. 
\]

Suppose $S$ is a complete intersection surface in $(\PP^1)^4$ obtained by intersecting  hypersurfaces $A$ and $B$. 
There are nonnegative integers $a_i$ and $b_i$ such that
\[
[ A]=a_1H_1+a_2H_2+a_3H_3+a_4H_4, \quad 
[ B]=b_1H_1+b_2H_2+b_3H_3+b_4H_4,
\]
If $[S]$ is equal to $\eta(\mathbf{p})$, then
\[
[ S]=\sum_{i<j} q_{ij}H_iH_j=\sum_{i<j} (a_ib_j+a_jb_i)H_iH_j.
\]
Therefore, we obtain, for example, 
\[
p_{14}p_{23}+p_{13}p_{24}-p_{12}p_{34}=
q_{14}q_{23}+q_{13}q_{24}-q_{12}q_{34}=2(a_1a_2b_3b_4+b_1b_2a_3a_4)\geq 0.
\]    
This proves the necessity of the condition $\mathbf{p} \in \Delta(\mathbb{T}_1)$.

For the converse, suppose that $\mathbf{p}$ is an integral point in  $\Delta(\mathbb{T}_1)$.
We use the following factorization in 
 $\H^4((\PP^1)^4,\mathbb{Q})$:
\begin{align*}
 &(2q_{14}q_{24}q_{34}) \cdot  (q_{12} H_1 H_2+q_{13} H_1 H_3+q_{14} H_1 H_4+q_{23} H_2 H_3+q_{24} H_2 H_4+q_{34} H_3 H_4)=\\& \quad (q_{14}H_1+q_{24}H_2+q_{34}H_3)\cdot\left(q_{14|23}q_{14}H_1+q_{13|24}q_{24}H_2+q_{12|34}q_{34}H_3+2q_{14}q_{24}q_{34}H_4\right).
\end{align*}
 The coefficients $q_{ij}$ and $q_{ij|kl}$ are nonnegative because $\mathbf{p} \in \Delta(\mathbb{T}_1)$.
Therefore, if 
$q_{14}q_{24}q_{34}$
is nonzero, then 
$2q_{14}q_{24}q_{34}\,\eta(\mathbf{p})$ 
is a complete intersection class.

 It remains to consider the cases when 
 $q_{il}q_{jl}q_{kl}=0$ for all $i,j,k,l$.
 Up to symmetry, the analysis reduces to two cases: 
 The case when 
 $q_{14}=q_{24}=q_{34}=0$ 
 and the case when $q_{12}=q_{34}=0$.

In the first case, we may assume  without loss of generality that 
$q_{12}>0$.
We have
\[
q_{12}\cdot(q_{12}H_1H_2+q_{13}H_1H_3+q_{23}H_2H_3)=(q_{12}H_1+q_{23}H_3)\cdot(q_{12}H_2+q_{13}H_3),
 \]
 so 
 $q_{12}\,\eta(\mathbf{p})$
 is a complete intersection class.

 In the second case,
we have $q_{13}q_{24}=q_{14}p_{23}$. Thus, we may assume that $q_{13},q_{24}, q_{23}, q_{14}$ are nonzero, since otherwise we reduce to the previous case. We have
     \[
     q_{13}\cdot (q_{13} H_1 H_3+q_{14} H_1 H_4+q_{23} H_2 H_3+q_{24} H_2 H_4)
     =(q_{13}H_1+q_{23}H_2) \cdot (q_{13}H_3+q_{14}H_4),
     \]
      so 
      $q_{13}\,\eta(\mathbf{p})$
      is a complete intersection class.
\end{proof}

\begin{remark}
Characterizing complete intersection classes in $\H_{2k}(\PP^\mathbf{m},\QQ)$ or $\H_{2k}(\PP^\mathbf{m},\RR)$ is an elementary but challenging problem.
For example, we do not know how to characterize complete intersection classes up to a positive multiple in $\H_{4}((\PP^2)^3,\RR)$.
If $\eta$ is a complete intersection surface class in $(\PP^2)^3$, then
\[
p_{12}p_{13}p_{23} \geq 4p_{11}p_{22}p_{33} \ \ \text{and} \ \ 
p_{jk}p_{ik} \geq  p_{ij}  p_{kk} \ \ \text{and} \ \
p_{ij}^2 \geq 2p_{ii}p_{jj},
\]
where
$p_{ij}=\int_S H_iH_j$ for all $i,j$.
Which other inequalities hold for complete intersection classes?
\end{remark}

\begin{remark}
Let $A_1,A_2,A_3,A_4$ be lattice polygons in $\RR^2$. 
If $f_i$ and $g_i$ are general Laurent polynomials with Newton polygon $A_i$, then the closure of the image of the map
\[
(f_1,g_1)\times(f_2,g_2)\times(f_3,g_3)\times  (f_4,g_4):(\mathbb{C}^*)^2 \longrightarrow (\PP^1)^4
\]
has the homology class given by $\lambda^{-1}(p_{12},p_{13},p_{14},p_{23},p_{24},p_{34}) \in \medwedge^2 \ZZ^4$ for some $\lambda>0$, where
\[
p_{ij}=\MV(A_i,A_j) \ \ \text{for all $i<j$.}
\]
In \cite[Theorem 3.2]{Averkov-Soprunov}, Averkov and Soprunov show that 
\[
\Delta(\mathbb{T}_1)= \left\{\mathbf{p}\in \medwedge^2\RR^4 \;\middle|\;  \text{$p_{ij}=\MV(A_i,A_j)$ for convex bodies $A_1,A_2,A_3,A_4 \subseteq \RR^2$}\right\}.
\]
This gives an additional geometric interpretation of the locus $\Delta(\mathbb{T}_1)$.
\end{remark}

\section{Realizable classes in products of projective spaces}\label{sec:manyProjectiveSpaces}

As before, let $\mathbf{m}$ be a vector of positive integers $(m_1,\ldots,m_n)$, and let  $\mathbb{P}^\mathbf{m}$ be the product of projective spaces $\prod_{i=1}^n \mathbb{P}^{m_i}$.
In this section, we prove Theorem~\ref{thm:nondegenerateSteenrod}:
A class $\eta \in \H_4(\mathbb{P}^\mathbf{m},\mathbb{Q})$ is realizable over $\QQ$ if and only if $L(\eta)$ is Lorentzian, where $L(\eta)$ is 
the $n \times n$ symmetric matrix with entries
 \[
L(\eta)_{ij}=\int_\eta H_iH_j \ \ \text{for $1 \le i \le j \le n$.}
 \]
In \cite{Sturmfels25}, the set of traceless Lorentzian matrices was identified with a massless Mandelstam region and parametrized by the future timelike cone in Minkowski space. 
To adapt this observation to the setting of Theorem~\ref{thm:nondegenerateSteenrod}, we construct in Theorem~\ref{prop:exists surface} a smooth projective surface of given Picard rank  that contains no negative curves, and on which every nonzero nef divisor is numerically equivalent to a semiample divisor.\footnote{A Cartier divisor on $X$ is said to be \emph{nef} if it intersects 
every irreducible curve in $X$ nonnegatively, and \emph{semiample} if some positive multiple of it is basepoint-free.}
In Theorem~\ref{thm:realization over Z},
we prove a more precise version of Theorem~\ref{thm:nondegenerateSteenrod} for $(\PP^1)^n$ with $2 \le n \le 11$: If $L(\eta)$ is an integral Lorentzian matrix with coprime entries and $p$ is a prime number,
then $\eta$ is realizable over $\ZZ$ or $p\eta$ is realizable over $\ZZ$.

\subsection{Prolific surfaces}\label{sec:surfaces no neg curves}

Consider the vector space $\RR^{1+d}$ with coordinate functions $x_0, x_1, \ldots, x_d$ equipped with the nondegenerate symmetric bilinear pairing
\[
\mathbf{x} \cdot \mathbf{y}=(x_0,x_1,\ldots,x_d) \cdot (y_0,y_1,\ldots,y_d)=x_0y_0-x_1y_1-\cdots-x_dy_d.
\]
The \emph{future timelike cone} $\mathcal{C}$ is the subset of vectors $\mathbf{x}$ satisfying $\mathbf{x} \cdot \mathbf{x} \geq 0$ and $x_0 \ge 0$.
The future timelike cone is a self-dual convex cone:
\[
\mathcal{C}=\left\{\, \mathbf{x} \in \RR^{1+d} \;\middle\vert\; \mathbf{x} \cdot \mathbf{y}\geq 0 \ \ \text{for all $\mathbf{y} \in \mathcal{C}$} \, \right\}.
\]
Thus, $\mathcal{C}$ consists of $\mathbf{x}$ satisfying $\mathbf{x} \cdot \mathbf{x} \geq 0$ and $\mathbf{x} \cdot \mathbf{y} \geq 0$
for any given $\mathbf{y}$ in the interior of $\mathcal{C}$.

We consider the analogous cone for a smooth projective surface $Y$. For any unexplained terms, we refer to \cite{Lazarsfeld1}. 
The intersection product defines a nondegenerate symmetric bilinear pairing on the real N\'eron--Severi space $\NS(Y)_\RR$ with exactly one positive eigenvalue. We consider the \emph{future timelike cone}
\[
\mathcal{C}(Y)\coloneq \left\{\, D \in \NS(Y)_\RR \,  \;\middle\vert\; \, \int_Y D^2 \geq 0 \ \ \text{and} \ \ \int_Y DH \ge 0 \,\right\},
\]
where $H$ is the class of a fixed ample divisor on $Y$. 
The cone $\mathcal{C}(Y)$ contains the ample cone, and it is independent of the choice of $H$.


In general, determining the cone of nef divisors on a smooth projective surface with large Picard number can be challenging. 
An important exception occurs when the surface does not contain any \emph{negative curve}, that is, a reduced and irreducible curve with negative self-intersection number.

\begin{lemma}\label{lem:C=Nef=NE}
    If $Y$ does not contain any negative curve, then the cone of nef divisors on $Y$ is equal to $\mathcal{C}(Y)$.
\end{lemma}

\begin{proof}
In any case, the nef cone is contained in $\mathcal{C}(Y)$. 
If $Y$ does not contain any negative curve, then the nef cone is self-dual with respect to the intersection pairing \cite[Example 1.4.33]{Lazarsfeld1}, so it must be equal to the self-dual cone $\mathcal{C}(Y)$.
\end{proof}

\begin{definition}
We say that a smooth projective surface $Y$ is \emph{prolific} if $Y$ does not contain any negative curve and every nef divisor on $Y$ is numerically equivalent to a semiample divisor.
\end{definition}

Let $\PP^\mathbf{m}$ be a product of $n$ projective spaces, and let $Y$ be a prolific surface of Picard rank at least $n+3$. 
We will show in the next subsection that
every realizable class in $\H_4(\PP^\mathbf{m},\QQ)$ is a rational multiple of $\varphi_* [Y]$ for some $\varphi: Y \rightarrow \PP^\mathbf{m}$.

\begin{theorem}\label{prop:exists surface}
For any positive integer $k$, there is a prolific surface $Y$ with Picard rank $k$.
\end{theorem}


It is enough to consider the case $k \ge 3$. We write $k=n+2$ for a positive integer $n$.
For the remainder of this subsection, 
 fix an elliptic curve $(E,e)$ with $\textrm{End}(E) \simeq \ZZ$ and an integer $d \ge 3$. 
Let $H_0,H_1,\ldots,H_n$ be the inverse images of $e$ under the coordinate projections $E^{n+1} \to E$. 
The $0$-th coordinate will play a distinguished role, and we will occasionally emphasize this by writing $E^{n+1}=E \times E^n$.
We set
\[
Y \coloneq \Bigg(\text{the intersection of $(n-1)$ general members in the linear system of $\sum_{i=1}^n dH_i$}\Bigg).
\]
The surface $Y$ is of the form $E \times C$, where $C$ is a complete intersection curve in $E^n$. By Bertini's theorem, $C$ is smooth and irreducible. We show that $Y$ has the required properties.

\begin{lemma}
The surface $Y$ does not contain any negative curve.
\end{lemma}

\begin{proof}
Since the elliptic curve $E$ acts on $Y$ by addition on the first factor, the only possible rigid curves are fibers of the projection $Y \to C$. However, any such curve moves in an algebraic family over $C$, so $Y$ does not contain any rigid curve.
\end{proof}

\begin{lemma}\label{lem:ellipticproduct}
For distinct indices $i$ and $j$, consider the divisor
\[
\Delta_{ij}\coloneq \left\{(x_0,x_1,\ldots,x_n) \in E^{n+1} \;\middle\vert\; x_i=x_j \right\}.
\]
Then  $[H_i]$ for all $i$ and $[\Delta_{ij}]$ for all distinct $i$ and $j$ form a basis of $\textrm{NS}(E^{n+1})_\QQ$.
\end{lemma}

\begin{proof}
Let $p_1$ and $p_2$ be the projections from $E \times E^n$ to $E$ and $E^n$, and let $q_1$ and $q_2$ be the inclusions from $E \simeq E \times e^n$ and $E^n\simeq e \times E^n$ to $E \times E^n$.
By \cite[Theorem 3.9]{Scholl}, we have
\[
\textrm{NS}(E \times E^{n})_\QQ / ( p_1^* \textrm{NS}(E)_\QQ + p_2^* \textrm{NS}(E^n)_\QQ ) \simeq \text{Hom}(E,E^n)\otimes \QQ.
\]
The pullbacks $p_1^*,p_2^*,q_1^*,q_2^*$ define and split the short exact sequence
\[
0 \longrightarrow \textrm{NS}(E)_\QQ \oplus  \textrm{NS}(E^n)_\QQ \longrightarrow \textrm{NS}(E \times E^{n})_\QQ \longrightarrow \textrm{Hom}(E,E^n)\otimes \QQ \longrightarrow 0.
\]
Since $\dim \textrm{Hom}(E,E^n)\otimes \QQ=n$ by the assumption on $E$, 
 by induction, we see that 
\[
\dim \textrm{NS}(E \times E^{n})_\QQ
=1+n+\frac{1}{2}n(n+1)=\frac{1}{2}(n+1)(n+2).
\]
To see the linear independence of $[H_i]$ and $[\Delta_{ij}]$, we intersect the divisor classes against the coordinate elliptic curves $E_k \simeq E$ and the diagonal elliptic curves $E_{kl} \simeq E$ in $E^{n+1}$ for all $k \neq l$: 
\begin{align*}
\int_{E_k} H_i&=\begin{cases}
    1&\text{if $i=k$},\\
    0&\text{otherwise},
\end{cases}\\
\int_{E_{kl}} H_i&=\begin{cases}
    1&\text{if there are exactly two distinct elements in $i,k,l$},\\
    0&\text{otherwise},
\end{cases}\\
\int_{E_k} \Delta_{ij}&=\begin{cases}
    1&\text{if there are exactly two distinct elements in $i,j,k$},\\
    0&\text{otherwise},
\end{cases}\\
\int_{E_{kl}} \Delta_{ij}&=\begin{cases}
    1&\text{if there are exactly three distinct elements in $i,j,k,l$}\\
    0&\text{otherwise}.
\end{cases}
\end{align*}
For any rational numbers $a_i$ and $b_{ij}$,
we see that the conditions 
\begin{align*}
\int_{E_k} \Bigg( \sum_i a_i H_i+\sum_{i<j} b_{ij} \Delta_{ij} \Bigg)&=a_k +\sum_{j\neq k} b_{jk}=0,\\
\int_{E_l} \Bigg( \sum_i a_i H_i+\sum_{i<j} b_{ij} \Delta_{ij} \Bigg)&=a_l +\sum_{j\neq l} b_{jl}=0,\\
\int_{E_{kl}} \Bigg( \sum_i a_i H_i+ \sum_{i<j} b_{ij} \Delta_{ij} \Bigg)&=a_k +a_l +\Bigg(-b_{kl}+\sum_{j \neq k} b_{jk}\Bigg) +\Bigg(-b_{kl}+\sum_{j \neq l} b_{jl}\Bigg)=0
\end{align*}
together imply $b_{kl}=0$ for all distinct $k$ and $l$, which in turn implies $a_k=0$ for all $k$.
\end{proof}

For the remainder of this subsection, we set $\gamma \coloneq d^{n-1}(n-1)!$ and 
$\Delta_i \coloneq \Delta_{0i}$ for $i \neq 0$.

\begin{lemma}\label{lem:intersection}
For distinct indices $i\neq 0$ and $j\neq 0$, we have
\[
\int_Y \Delta_{i} \Delta_{j}=2\gamma, \qquad \int_Y  \Delta_{i} H_0 = \int_Y \Delta_{i} H_1 =\gamma, 
\qquad \int_Y  H_0 H_1 = \gamma.
\]
\end{lemma}

\begin{proof}
Since $H_0$, $H_i$, and $\Delta_i$ intersect transversely, we have
\[
[H_i] [\Delta_{i}]=[H_0] [\Delta_{i}]=[H_i][H_0] \ \ \text{and} \ \ [\Delta_i][\Delta_i]=[H_i][H_i]=[H_0][H_0]=0.
\]
Thus, for example, we have
\begin{align*}
\int_Y \Delta_{1} \Delta_{2} &=d^{n-1} \int_{E^{n+1}} \Delta_1 \Delta_2 (H_1+\cdots+H_n)^{n-1}\\
&= \gamma \left[ \int_{E^{n+1}} \Delta_{1}\Delta_{2} H_1 (H_3 \cdots H_n)+\int_{E^{n+1}} \Delta_{1} \Delta_{2} H_2 (H_3 \cdots H_n)\right] \\
&= \gamma \left[ \int_{E^{n+1}} \Delta_{2} H_0 H_1 (H_3 \cdots H_n) +\int_{E^{n+1}} \Delta_{1} H_0 H_2(H_3 \cdots H_n)\right]\\
&=\gamma \left[ \int_{E^{n+1}} H_0H_1H_2(H_3 \cdots H_n)+\int_{E^{n+1}} H_0H_1H_2(H_3 \cdots H_n) \right]=2\gamma.
\end{align*}
The other computations are similar and simpler. For example, we have
\begin{align*}
\int_Y  \Delta_{i} H_0 &=d^{n-1}\int_{E^{n+1}} H_0H_i (H_1+ \cdots+ H_n)^{n-1}=\gamma \int_{E^{n+1}} H_0H_1H_2 \cdots H_n
=\gamma, \\
\int_Y  \Delta_{i} H_1 &=d^{n-1}\int_{E^{n+1}} \Delta_i H_1 (H_1+ \cdots+ H_n)^{n-1}=\gamma \int_{E^{n+1}} \Delta_i H_1H_2 \cdots H_n
=\gamma. \qedhere
\end{align*}
\end{proof}

We use the computation in Lemma~\ref{lem:intersection} to obtain an explicit basis of $\textrm{NS}(Y)_\QQ$. 
We write $\iota^*$ for the pullback map $\textrm{NS}(E \times E^n)_\QQ \to \textrm{NS}(Y)_\QQ$.

\begin{lemma}\label{lem:basis}
The N\'eron--Severi space $\NS(Y)_\QQ$ has basis
\[
[D_1]\coloneq \iota^*[\Delta_{1}], \ \ \ldots, \ \ [D_n]\coloneq \iota^*[\Delta_{n}], \quad [D_{n+1}]\coloneq  \iota^*[H_0], \quad [D_{n+2}]\coloneq \iota^*[H_1].
\]
\end{lemma}

\begin{proof}
Recall that $Y =E \times C$ for $C \subseteq E^n$.
Arguing as before for distinct $i \neq 0$ and $j \neq 0$, 
\[
\int_C \Delta_{ij}=\int_C 2H_i=2\gamma,  \ \ \text{and hence} \ \ 
\iota^* [\Delta_{ij}] =2\iota^* [H_i]= 2 [D_{n+2}].
\]
Lemma~\ref{lem:ellipticproduct} thus implies that
 $[D_1],\ldots,[D_{n+2}]$ span $\textrm{im}(\iota^*)$. 

Each $[D_i]$ is either a pullback of the diagonal in $E \times E$ or a pullback of a divisor on $E$, so its self-intersection is zero. Therefore, by Lemma~\ref{lem:intersection}, the intersections between
\[
 [D_i]-[D_{n+1}]-[D_{n+2}]  \quad [D_{n+1}]-[D_{n+2}], \quad [D_{n+1}]+[D_{n+2}], \quad i=1,\ldots,n,
\]
are given by the diagonal matrix with diagonal entries
$-2\gamma,\ldots,-2\gamma,2\gamma$. For example,
\begin{align*}
& \Big( [D_1]-[D_{n+1}]-[D_{n+2}] \Big) \Big(  [D_2]-[D_{n+1}]-[D_{n+2}] \Big)= 0, \\
&  \Big( [D_1]-[D_{n+1}]-[D_{n+2}] \Big) \Big(  [D_{n+1}]-[D_{n+2}] \Big)= 0, \\
&   \Big( [D_1]-[D_{n+1}]-[D_{n+2}] \Big) \Big(  [D_{n+1}]+[D_{n+2}] \Big)= 0.
\end{align*}
Since this matrix is nondegenerate,  $[D_1],\ldots,[D_{n+2}]$ form a basis of $\textrm{im}(\iota^*)$.

It remains to show that $\textrm{im}(\iota^*)=\NS(Y)_\QQ$.
The key step in the proof is the following classical result, going back to Severi\footnote{Severi's original arguments in \cite{Severi1} and \cite{Severi2} were incomplete. A complete proof was later provided by Ciliberto and van der Geer in \cite{Ciliberto-Geer}. The same statement was generalized to arbitrary uncountable algebraically closed ground fields in \cite{Banerjee}. When the ground field is $\overline{\QQ}$, the existence of the curve $C$ was established in \cite{Koch}.}:
\begin{quote}
\emph{Let $S$ be a smooth projective surface. If $C$ is a very general curve in the linear system of a very ample divisor, then the kernel of the natural map $\textrm{Jac}(C) \to \textrm{Alb}(S)$ is a simple abelian variety.}
\end{quote}
Let $S$ be the intersection of $(n-2)$ general members in the linear system of $\sum_{i=1}^n dH_i$ in $E^n$.
By the Lefschetz hyperplane theorem, 
 the Albanese variety of $S$ is isomorphic to $E^n$,\footnote{For the case over an arbitrary algebraically closed ground field, see \cite[Corollaire XII.3.6]{SGA2}.}
 and hence the Jacobian of $C$ is isogenous to $A \times E^n$ for some simple abelian variety $A$ of large dimension.
 Therefore, by \cite[Theorem 3.9]{Scholl}, the surface $Y=E\times C$ satisfies
    \[
\textrm{NS}(Y)_\QQ /(\textrm{NS}(E)_\QQ + \textrm{NS}(C)_\QQ) \simeq  \textrm{Hom}(E,\textrm{Jac}(C)) \otimes \QQ \simeq   \textrm{Hom}(E,E^n) \otimes \QQ.
    \]
It follows that the dimension of $\NS(Y)_\QQ$ is equal to that of $\textrm{im}(\iota^*)$.
\end{proof}

\begin{proof}[Proof of Theorem~\ref{prop:exists surface}]

It remains to show that any nef divisor $D$ is numerically equivalent to a semiample divisor.
By replacing $D$ with a positive multiple if necessary, we may write
\[
[D]=d_1[D_1]+\cdots+d_{n+1}[D_{n+1}]+d_{n+2}[D_{n+2}] \ \ \text{for} \ \  d_i \in \ZZ.
\]
By Lemma~\ref{lem:intersection}, 
 the integers $d_i$ are determined by the intersection numbers
\[
\frac{1}{\gamma}\int_Y D  D_i=\begin{cases} \hfill 2s_0-2d_i+d_{n+1}+d_{n+2} \hfill& \text{when $1 \le i \le n,$}\\
\hfill s_0+d_{n+2}\hfill & \text{when $i=n+1,$}\\
 \hfill s_0+d_{n+1} \hfill& \text{when $i=n+2,$}
 \end{cases}
\]
where $s_0\coloneq d_1+\cdots+d_n$. 
If the self-intersection of $D$ is positive, then $D$ is ample by Lemma~\ref{lem:C=Nef=NE}. Thus, we may and will assume that the self-intersection of $D$ is zero:
\[
\frac{1}{2\gamma} \int_Y D^2 
 =\Big(s_0+d_{n+1}\Big)\Big(s_0+d_{n+2}\Big)-\Big(\sum_{j=1}^n d_j^2\Big)=0.
\]
Since $D$ is nef, the integer $s_0+d_{n+1}$ is nonnegative. If it is zero, then $d_1=\cdots=d_n=0$ by the displayed formula, so $D$ is numerically equivalent to a nonnegative multiple of the semiample divisor $D_{n+2}$.
Thus, without loss of generality, we may suppose that $s_0+d_{n+1}$ is positive.

We use the group structure of $E$ to define a map
\[
\Psi:E \times E^{n} \longrightarrow E, \qquad (x_0,x_1,\ldots,x_n) \longmapsto (s_0+d_{n+1}) x_0-d_1x_1+\cdots-d_nx_n.
\]
We gather some intersection-theoretic information on the fiber $F$ of $\Psi$ over $e$:
\begin{enumerate}[(1)]\itemsep 5pt
\item When $i$ is an element of $\{1, \ldots, n\}$, the intersection product of $F$ and $(H_1 \cdots H_{i-1}) \Delta_i (H_{i+1} \cdots H_n)$ counts the number of points $x_0$ in $E$ such that $(s_0+d_{n+1}-d_i) x_0$ is a given general point. Since there are $(s_0+d_{n+1}-d_i)^2$ such points, we have
\[
\int_{E^{n+1}} F (H_1 \cdots H_{i-1}) \Delta_i (H_{i+1} \cdots H_n) = (s_0+d_{n+1}-d_i)^2.
\]
\item When $j$ is an element of $\{1, \ldots, n\}$, the intersection product of $F$ and $(H_1 \cdots H_{j-1})H_0 (H_{j+1} \cdots H_n)$ counts the number of points $x_j$ in $E$ such that $-d_j x_j$ is a given general point. Since there are $d_j^2$ such points, we have
\[
\int_{E^{n+1}} F (H_1 \cdots H_{j-1}) H_0 (H_{j+1} \cdots H_n) = d_j^2.
\]
\item When $i$ and $j$ are distinct elements of $\{1, \ldots, n\}$, the intersection product of $F$ and\\
$(H_1 \cdots H_{j-1}) \Delta_i (H_{j+1} \cdots H_n)$ counts the number of points $x_j$ in $E$ such that $-d_j x_j$ is a given general point. Since there are $d_j^2$ such points, we have
\[
\int_{E^{n+1}} F (H_1 \cdots H_{j-1}) \Delta_i (H_{j+1} \cdots H_n) = d_j^2.
\]
\end{enumerate}
Thus, for $1 \le i \le n$, we have
\begin{align*}
&\ \ \, \frac{1}{\gamma}\int_Y  \iota^*(F) \,  D_{i} \ \,   = \frac{1}{(n-1)!}\int_{E^{n+1}} F \Delta_i (H_1 +\cdots+H_n)^{n-1}=  (s_0+d_{n+1}+d_i)^2- d_i^2+  \sum_{j=1}^n d_j^2, \\
&\frac{1}{\gamma}\int_Y  \iota^*(F)  D_{n+1} = \frac{1}{(n-1)!}\int_{E^{n+1}} F H_0 (H_1 +\cdots+H_n)^{n-1}= \sum_{j=1}^n d_j^2,\\
&\frac{1}{\gamma}\int_Y  \iota^*(F)  D_{n+2} = \frac{1}{(n-1)!}\int_{E^{n+1}} F H_1 (H_1 +\cdots+H_n)^{n-1}= (s_0+d_{n+1})^2.
\end{align*}
Comparing the intersection numbers for $\iota^*(F)$ with the intersection numbers for $D$, we see that
\[
\frac{1}{\gamma}\int_Y  \iota^*(F)  D_{i}=\frac{1}{\gamma}\int_Y  (s_0+d_{n+1})D D_{i} \ \ \text{for $1 \le i \le n+2$.}
\]
Since the intersection pairing is nondegenerate, 
$(s_0+d_{n+1})D$ is numerically equivalent to $\iota^*(F)$, which is semiample.
\end{proof}

\begin{remark}\label{rmk:Picard rank}
The proof of Lemma~\ref{lem:basis}, and hence the proof of Theorem~\ref{prop:exists surface}, is valid if the ground field to be uncountable or of characteristic $0$.
Is there a prolific surface $Y$ over $\overline{\mathbb{F}}_p$ of any given Picard rank? 
\end{remark}

\subsection{Classes realizable over $\QQ$}\label{subsection:rational}

We use a prolific surface $Y$ of large Picard rank to  prove Theorem~\ref{thm:nondegenerateSteenrod}.
The key remaining tool is Meyer's theorem on indefinite rational quadratic forms \cite[Chapter 6]{Cassels}:
\begin{quote}
\emph{If a homogeneous quadratic equation $Q(x_1,\ldots,x_n)=0$ with rational coefficients in five or more variables has a nonzero real solution, then it has a nonzero rational solution.}
\end{quote}
For a rational quadratic form $Q$, we write $(s_+(Q),s_-(Q),s_0(Q))$ for the signature of $Q$.

\begin{lemma}\label{lem:lattice embedding}
    Let $P$ be a rational quadratic form on $\QQ^m$,
    and let $Q$ be a rational quadratic form on $\QQ^n$ such that
    \[
s_+(P) \le s_+(Q), \ \
s_-(P) \le s_-(Q), \ \ \text{and} \ \ 
s_+(P)+s_-(P)+3 \le s_+(Q)+s_-(Q).
    \]
Then there is a $\QQ$-linear map $f:\QQ^m \to \QQ^n$ such that $P(x)=Q(f(x))$ for all $x \in \QQ^m$.
\end{lemma}

\begin{proof}
By taking a complement of the kernel of $Q$ in $\QQ^n$,  we can reduce to the case when $Q$ is nondegenerate.
We prove the statement by induction on $m$ under this assumption, the case $m=0$ being trivial.

Suppose $P$ is not identically zero, since otherwise the assertion is trivial. In this case, $n \ge 4$,  and there is $x_0 \in \QQ^m$ such that $P(x_0)$ is nonzero. 
We introduce a new variable $z$, and apply Meyer's theorem to the quadratic equation 
\[
R(y,z) \coloneq Q(y)-P(x_0)z^2=0.
\]
Since $Q$ is nondegenerate, $R$ is nondegenerate, so the existence of a rational solution for $R=0$ implies the existence of a dense set of rational solutions for $R=0$. We choose a nonzero $q \in \QQ$ and $y_0 \in \QQ^n$ such that
\[
Q(y_0)=P(x_0)q^2=P(q x_0).
\]

Let $x_0^\perp$ be the orthogonal complement to $x_0$ in $\QQ^m$,
and let $y_0^\perp$ be the orthogonal complement to $y_0$ in $\QQ^n$.
Since $Q(y_0)$ and $P(x_0)$ have the same sign, the induction hypothesis applies to the restrictions of $P$ and $Q$ to these orthogonal complements. The outcome is that there is a $\QQ$-linear map $f: x_0^\perp \to y_0^\perp$ such that $P(x)=Q(f(x))$ for all $x \in x_0^\perp$. 
The unique $\QQ$-linear extension $f:\QQ^m \to \QQ^n$ satisfying $f(qx_0)=y_0$ has the desired property.
\end{proof}

\begin{proof}[Proof of Theorem~\ref{thm:nondegenerateSteenrod}]
Let $\mathbf{m}=(m_1,\ldots,m_n)$ be a vector of positive integers, and let $\eta \in \H_4(\PP^\mathbf{m},\QQ)$ be a class with $L(\eta)$ Lorentzian. 
We show that some rational multiple of $\eta$ is realizable over $\ZZ$.

By Theorem~\ref{prop:exists surface}, there is a prolific surface $Y$ of Picard rank $n+3$. 
By Lemma~\ref{lem:lattice embedding}, we can find  $[A_1],\ldots,[A_n] \in \NS(Y)_\QQ$ such that $L(\eta)_{ij}=\int_Y A_i A_j$ for all $i$ and $j$.
Choose an ample divisor $H$ on $Y$ such that $\int_Y A_i H$ is nonzero for all $i$, and set 
\[
[B_i]\coloneq \begin{cases}
    +[A_i]& \text{if $\int_YA_i H$ is positive},\\
    -[A_i]& \text{if $\int_Y A_i H$ is negative}.
\end{cases}
\]
Since $Y$ does not contain any negative curve, Lemma~\ref{lem:C=Nef=NE} determines the nef cone of $Y$. We deduce that $[B_i]$ is nef for every $i$ because $\int_Y B_i H$ and $\int_Y B_i^2$ are nonnegative. 
Since the intersection between any nef divisors is nonnegative, we have 
$L(\eta)_{ij}=\int_Y B_i B_j$ for all $i$ and $j$.

Since $Y$ is prolific, we may suppose that $B_i$ are semiample $\QQ$-divisors satisfying $L(\eta)_{ij}=\int_Y B_iB_j$ for all $i$ and $j$. 
Choose a positive integer $\ell$ such that 
each $\ell B_i$ is the pullback of a hyperplane under the map
\[
\varphi_i: Y \longrightarrow \PP \H^0(Y,\mathcal{O}_Y(\ell B_i))^\vee.
\]
The image of $\varphi_i$ is a curve if and only if $L(\eta)_{ii}=0$.

If the dimension of the target projective space is at most $m_i$, we replace $\varphi_i$ by the composition of $\varphi_i$ with a linear inclusion into $\PP^{m_i}$. If the dimension of the target projective space is larger than $m_i$, we replace $\varphi_i$ by the composition of $\varphi_i$ with a general linear projection onto $\PP^{m_i}$. 
Since the image of $\varphi_i$ is a curve whenever $m_i= 1$, such compositions do not change the class of the pullback of a hyperplane.
Combining all the maps, we get
\[
\varphi\coloneq \prod_{i=1}^n \varphi_i: Y \longrightarrow \PP^{\mathbf{m}}.
\]
By construction, $[\varphi(Y)]=\frac{\ell^2}{\lambda} \eta$, where $\lambda$ is the degree of the map $Y \to \varphi(Y)$.
\end{proof}

\begin{remark}
As noted before, Theorem~\ref{prop:exists surface} remains valid over any uncountable algebraically closed field. From this, one can deduce that Theorem~\ref{thm:nondegenerateSteenrod} remains valid over any algebraically closed field if we replace $\H_4(\PP^\mathbf{m},\QQ)$ by the Chow group $\textrm{CH}_2(\PP^\mathbf{m},\QQ)$.
To see this,
let $k_1$ be an algebraically closed field, and let $k_2$ be an uncountable algebraically closed field extension of $k_1$.
Suppose that $Z_2$ is an irreducible subvariety of $\PP^\mathbf{m}$ over $k_2$.
Fixing an ample divisor of $\PP^\mathbf{m}$, the subvariety $Z_2$ defines a $k_2$-point of the Hilbert scheme $\mathrm{Hilb}^P(\PP^\mathbf{m})$, where $P$ is the Hilbert polynomial of $Z_2$ with respect to the fixed ample divisor \cite{FGA}. 
The field extension $k_1 \subseteq k_2$ induces an injective map $\mathrm{Hilb}^P(\PP^\mathbf{m})(k_1)\to \mathrm{Hilb}^P(\PP^\mathbf{m})(k_2)$ whose image is Zariski dense. Since being geometrically integral is an open condition \cite[Th\'eor\`eme 12.2.1]{EGA4-2}, there is a $k_1$-point of $\mathrm{Hilb}^P(\PP^\mathbf{m})$ such that the corresponding $k_1$-subscheme $Z_1$ is integral and $[Z_1]=[Z_2]$. 
\end{remark}

Let $X$ be a $d$-dimensional irreducible projective variety over an algebraically closed field $k$.
For any collection of nef $\QQ$-divisors $D=(D_1,\ldots,D_n)$ on $X$,
consider the polynomial 
\[
f_H(x_1,\ldots,x_n)  = \frac{1}{d!}\int_X (x_1D_1+\cdots+x_nD_n)^{d}.
\]
A \emph{volume polynomial over $k$} is any limit of such polynomials. 
A volume polynomial over $k$ is a Lorentzian polynomial in the sense of \cite{Branden-Huh}. Theorem~\ref{thm:nondegenerateSteenrod} implies that the converse holds when $d=2$.

\begin{corollary}
For a quadratic polynomial $f$, the following conditions are equivalent:
\begin{enumerate}[(1)]\itemsep 5pt
\item $f$ is a volume polynomial over some $k$.
\item $f$ is a volume polynomial over any $k$.
\item $f$ is a Lorentzian polynomial.
\end{enumerate}
\end{corollary}

\begin{remark}\label{rmk:subtlety}
When $d \ge 3$, it is not known if the set of volume polynomials over $\mathbb{K}$ depends on $k$.
Not every cubic Lorentzian polynomial is a volume polynomial over $k$ for some $k$. The following example was given in \cite[Example 14]{HuhICM22}:
\[
f=14x_1^3+6x_1^2x_2+24x_1^2x_3+12x_1x_2x_3+6x_1x_3^2+3x_2x_3^2.
\]
This example highlights the subtlety involved in formulating a reasonable conjecture that characterizes, for example, the realizable classes in $\H_6((\PP^3)^3,\RR)$.
\end{remark}

\subsection{Classes realizable over $\ZZ$}\label{subsection:integral}

Let $n$ be an integer satisfying $2 \le n \le 11$. This will be the standing assumption in this subsection.
We use K3 surfaces to construct many classes in $\H_4((\PP^1)^n, \ZZ)$ that are realizable over $\ZZ$. 
An integral matrix is said to be \emph{primitive} if its entries are coprime, that is, when the greatest common divisor of the entries is $1$.

\begin{theorem}\label{thm:realization over Z}
Let $\eta$ be a class in $H_4((\PP^1)^n,\ZZ)$.  
If $L(\eta)$ is a primitive Lorentzian matrix and $p$ is a prime number, then $\eta$ is realizable over $\ZZ$ or $p\eta$ is realizable over $\ZZ$.
\end{theorem}




Unlike the other arguments in this paper, our proof of Theorem~\ref{thm:realization over Z} essentially relies on the fact that we are working over a field of characteristic $0$. We do not know if  Theorem~\ref{thm:realization over Z} remains valid if we work over any algebraically closed field or if we drop the assumption that $n \le 11$.

We deduce Theorem~\ref{thm:realization over Z} from the following existence result.

\begin{proposition}\label{prop:K3}
Let $L$ be a nondegenerate rank $n$ integral Lorentzian matrix with zero diagonal entries.  
If the quadratic form $\mathbf{x}L\mathbf{x}^T$ does not take the value $-2$ on $\ZZ^n$, 
then there is a complex K3 surface $Y$ and divisors $D_1,\ldots,D_n$ on $Y$ satisfying the following conditions:
\begin{enumerate}[(1)]\itemsep 5pt
\item The intersection product $\int_Y D_iD_j$ is equal to $L_{ij}$ for all $i$ and $j$.
\item The classes $[D_1],\ldots,[D_n]$ form a $\ZZ$-basis of the N\'eron--Severi group  $\textrm{NS}(Y)_\ZZ$.
\item The surface $Y$ does not contain any negative curve.
\item The divisors $D_1,\ldots,D_n$ are basepoint-free.
\end{enumerate}
%
\end{proposition}

\begin{proof}
The main ingredient is \cite[Corollary 2.9, Remark 2.11]{Morrison}: 
\begin{quote}
\emph{If $n \le 11$, then every even lattice of signature $(1,n-1)$ occurs as the N\'eron--Severi group of some complex algebraic K3 surface equipped with the intersection pairing.}
\end{quote}
Its proof relies on the global surjectivity of the period mapping for K3 surfaces \cite[Section 7.4]{Huybrechts} and Nikulin's lattice embedding theorem \cite[Theorem 1.10.1]{Nikulin}.

Since the diagonal entries of $L$ are zero, the quadratic form $\mathbf{x}L\mathbf{x}^T$ is even. We thus get a complex K3 surface $Y$ and divisors $D_1,\ldots,D_n$ satisfying conditions ($1$) and ($2$).
Any negative curve on a K3 surface must have self-intersection $-2$ \cite[Section 2.1.3]{Huybrechts}, so the surface $Y$ must also satisfy condition ($3$).

Choose an ample divisor $H$ on $Y$ such that $\int_Y D_i H$ is nonzero for all $i$. Since $Y$ does not contain any negative curve, Lemma~\ref{lem:C=Nef=NE} determines the nef cone of $Y$. We deduce that  either $D_i$ is nef or $-D_i$ is nef for each $i$, because the self-intersections are zero.
By replacing every $D_i$ by $-D_i$ if necessary, we may suppose without loss of generality that $D_1$ is nef. We show in this case that $D_i$ is nef for every $i$. This gives condition (4), because every nef divisor $D$ of self-intersection zero on a K3 surface is basepoint-free \cite[Chapter 2, Proposition 3.10]{Huybrechts}.

Suppose $D_1,\ldots,D_k$ and $-D_{k+1},\ldots,-D_n$ are nef divisors. If $D_i$ is chosen from the first group and $D_j$ is chosen from the second group, then $L_{ij}=\int_Y D_i D_j$ is nonpositive, so $L_{ij}$ must be zero.
This means that $L$ is a symmetric block-diagonal matrix with nonnegative entries, so $L$ has at least two positive eigenvalues if the decomposition is nontrivial. 
Since $L$ is Lorentzian, the block decomposition is trivial, so every $D_i$ is nef. 
\end{proof}

\begin{remark}
By \cite[Remark 17.2.16]{Huybrechts}, we may find such a K3 surface over $\overline{\QQ}$. This shows that Theorem~\ref{thm:realization over Z} remains valid over any algebraically closed field of characteristic $0$.
\end{remark}

\begin{proof}[Proof of Theorem \ref{thm:realization over Z}]
We prove the following more precise statement:
\begin{quote}
\emph{If $L(\eta)$ is a nonzero integral Lorentzian matrix and the quadratic form $\mathbf{x} L(\eta) \mathbf{x}^T$ does not take the value $-2$ on $\ZZ^n$, then
 $\lambda^{-1}\eta$ is realizable over $\ZZ$ for some positive integer $\lambda$.}
 \end{quote}
Applying this statement to $p\eta$ when $L(\eta)$ is primitive,
we see that $\lambda^{-1}\eta$ is realizable over $\ZZ$ for some positive integer $\lambda$, where $\lambda$ is necessarily $1$ or $p$.

We first consider the case when $L(\eta)$ is nondegenerate. By Proposition~\ref{prop:K3}, there is a complex K3 surface $Y$ and divisors $D_1, \ldots, D_n$ with the stated properties. We choose a pencil in the linear system of $D_i$ for each $i$ to define a map
\[
\varphi: Y \longrightarrow (\PP^1)^n.
\]
By construction, $[\varphi(Y)]=\lambda^{-1} \eta$, where $\lambda$ is the degree of the map $Y \to \varphi(Y)$.

When $L(\eta)$ is degenerate, consider the nondegenerate quadratic form $M(\eta)$ on $\ZZ^n/\ker L(\eta)$ induced by $L(\eta)$. By \cite[Corollary 2.9, Remark 2.11]{Morrison}, there is a K3 surface $Y$ 
and an isometry 
\[
 \ZZ^n/\ker L(\eta)\simeq \NS(Y)_\ZZ.
\]
Let $[D_1],\ldots,[D_n]$ be the divisor classes on $
Y$ corresponding to the image of the standard basis vectors of $\ZZ^n$, so that the intersection matrix of $D_1, \ldots, D_n$ is equal to  $L(\eta)$. As in the proof of Proposition \ref{prop:K3}, replacing every $D_i$ by $-D_i$ if necessary, we may assume that each $D_i$ is a nef divisor. By \cite[Chapter 2, Proposition 3.10]{Huybrechts}, each $D_i$ is basepoint-free, and we can apply the same construction as in the nondegenerate case. 
\end{proof}

\section{Realizability in Grassmannian and other varieties}\label{sec:Grass}

Let $\mathrm{B}_k(X) \subseteq \H_{2k}(X,\RR)$ be the subspace spanned by the $k$-dimensional algebraic cycles on $X$ with real coefficients. We consider the closed subset
\[
\mathscr{R}_k(X)\coloneq \{\text{$k$-dimensional classes on $X$ that are realizable over $\RR$}\}\subseteq \mathrm{B}_k(X).
\]
We write $\mathscr{R}^k(X)$ for $\mathscr{R}_{d-k}(X)$, where $d$ is the dimension of $X$.
In general, $\mathscr{R}_k(X)$ may exhibit sporadic structures in that it may be different from the closure of its interior in $\mathrm{B}_k(X)$.

\begin{proposition}
For any smooth projective variety $X$, we have    \[
    \mathcal{R}^1(X)^\circ=\mathrm{Mov}^1(X)^\circ,
    \]
    where $\mathrm{Mov}^1(X)^\circ$ denotes the interior of the movable cone of divisors in $\NS(X)_\RR$.\footnote{Recall that the \emph{movable cone of divisors} is the closure of the cone generated by classes of effective Cartier divisors $L$ such that the base locus of the complete linear system $|L|$ has codimension at least $2$. We thank Chenyang Xu for suggesting this statement.} 
\end{proposition}

In particular, when $X$ is a surface, the interior of $\mathcal{R}^1(X)$ is precisely the ample cone of $X$, and every ray in $\mathcal{R}^1(X) \setminus \overline{\mathcal{R}^1(X)^\circ}$ is isolated.

\begin{proof}
    Notice that when $[D]$ is in the interior of either $\mathrm{Mov}^1(X)$ or $\mathcal{R}^1(X)$, then $[D]$ is in the interior of the pseudoeffective cone, and hence $D$ is big \cite[Theorem 2.2.26]{Lazarsfeld1}.

    Suppose that $D$ is an irreducible divisor whose class is 
     in $\mathcal{R}^1(X)^\circ$. Since $D$ is big, there is a smallest integer $m$ such that  $mD$ is linearly equivalent to $E$ for some effective divisor $E$ different from $mD$. By the minimality of $m$, no component of $E$ is equal to $D$. Thus, the base locus of the linear system $|D|$ must have codimension at least $2$, and hence $D$ is movable. 

Suppose that $D$ is an effective divisor whose class is in $\mathrm{Mov}^1(X)^\circ$. Since $D$ is big, there is a positive integer $r$ such that 
the base locus of the linear system $|rD|$ has codimension at least $2$ and the induced rational map $\varphi: X\dashrightarrow \PP^n$ is generically injective. Let $U$ be the open subset of $X$ where $\varphi$ is well-defined, and let $H$ be a general hyperplane of $\PP^n$. Then, by Bertini's irreducibility theorem, the closure of $U \cap \varphi^{-1}H$
 is an irreducible divisor of $X$ linearly equivalent to $rD$, so $[D]$ is realizable over $\QQ$.
\end{proof}

It follows that the interior of $\mathcal{R}^1(X)$ is convex for any smooth projective variety $X$.
On the other hand, the interior of $\mathcal{R}^2(X)$ is typically not convex, as we have seen in the case when $X=(\PP^1)^4$.

In the following example, we observe that the interior of $\mathscr{R}_1(X)$ need not be connected.

\begin{example}
A general quintic threefold $Y$ in $\PP^4$ contains a line $L$ with the normal bundle 
\[
N_{L/Y}\cong \mathcal{O}(-1)\oplus \mathcal{O}(-1).
\]
See, for example, \cite[Proposition 6.30]{EisenbudHarris3264}. 
Let $X$ be the blow up of $Y$ along $L$, and denote the pullback of the hyperplane class  by $H$ and the exceptional divisor by $E$, so that
\[
E\simeq \PP^1\times \PP^1 \ \  \text{and}  \ \ N_{E/Y}\simeq \mathcal{O}_{\PP^1}(-1)\boxtimes\mathcal{O}_{\PP^1}(-1).
\]
The Chow ring of $X$ is determined by the relations
\[
    \int_{Y}H^3=5, \quad \int_{Y}H^2E=0, \quad \int_{Y}HE^2=-1, \quad  \int_{Y}E^3=2, \ \ \text{and} \ \  H^2+10HE+5E^2=0.
\]
The effective cone of divisors on $X$ is generated by $H-E$ and $E$, and its dual cone of movable curves on $X$ is given by
\[
\textrm{Mov}_1(X)=\textrm{cone}\Bigg(\frac{H^2}{5}, \frac{H^2}{5}-HE\Bigg) \subseteq \mathscr{R}_1(X).
\]
It is easy to check that the class of 
any irreducible curve in $X$ not contained in $E$ is in this cone. 
In general,  the cone of movable curves is the closure of the pushforwards of complete intersection of ample divisors under projective birational mappings \cite{BDPP13}. In this case, one can check that 
any class in this cone can be approximated by the classes of complete intersections of ample divisors on $X$. This gives one component of $\mathscr{R}_1(X)^\circ$. 
The other component of $\mathscr{R}_1(X)^\circ$ is given by the curves contained in $E\simeq \PP^1 \times \PP^1$:
\[
\iota_*\textrm{Mov}_1(E)=\textrm{cone}\Bigg(HE, \frac{H^2}{5}+HE\Bigg) \subseteq \mathscr{R}_1(X).
\]
Together they generate the effective cone of curves in $X$.
The shaded region in Figure~\ref{fig:curves} shows the classes of irreducible curves $C$ in $X$:
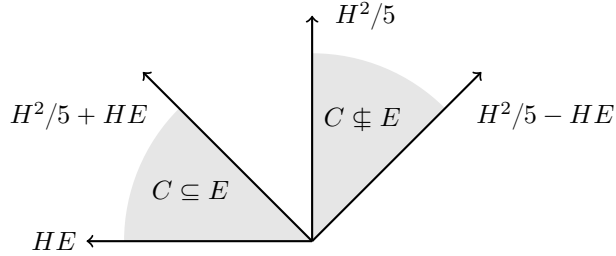
\begin{figure}[htbp]
  \centering
\[
\begin{tikzpicture}[scale=2.5]
  \fill[gray!20] (0,0) -- (-1,0) arc[start angle=180,end angle=135,radius=1] -- cycle;
  \node at ({0.7*cos(157.5)},{0.7*sin(157.5)}) {$C\subseteq E$};
  \fill[gray!20] (0,0) -- (0,1) arc[start angle=90,end angle=45,radius=1] -- cycle;
  \node at ({0.7*cos(67.5)},{0.7*sin(67.5)}) {$C\nsubseteq E$};
  \draw[->,thick] (0,0) -- (-1.2,0) node[left]{$HE$};
  \draw[->,thick] (0,0) -- (-0.9,0.9) node[pos=0.6,anchor=south east,xshift=-20pt]{$H^2/5+HE$};
  \draw[->,thick] (0,0) -- (0,1.2) node[above,right,xshift=5pt]{$H^2/5$};
  \draw[->,thick] (0,0) -- (0.9,0.9) node[pos=0.6,anchor=south west,xshift=20pt]{$H^2/5-HE$};
\end{tikzpicture}
\]
  \caption{The classes of irreducible curves in $X$}\label{fig:curves}
\end{figure}

\end{example}


The \emph{pseudoeffective cone} in $\mathrm{B}_k(X)$ is the closure of the cone generated by the $k$-dimensional effective cycles on $X$,
and the elements of this cone are called \emph{pseudoeffective}.
An element $\eta \in \mathrm{B}_k(X)$ is called \emph{universally pseudoeffective} if $\pi^*\eta$ is pseudoeffective for any morphism $\pi:Y \to X$ from a smooth projective variety $Y$. 
The cone of $k$-dimensional universally pseudoeffective classes on $X$ 
is a higher-codimension analogue of the cone of nef divisors.\footnote{
A $k$-dimensional universally pseudoeffective class is \emph{nef}: Its  intersection with any codimension $k$ subvariety of $X$ is nonnegative.  Grothendieck asked whether nef classes are always pseudoeffective (and thus universally pseudoeffective). Debarre, Ein, Lazarsfeld, and Voisin provided a negative answer to this question in \cite{DELV} when $k=2$.}
In \cite{Fulger-Lehmann}, Fulger and Lehmann show 
 several desirable properties of the cone of universally pseudoeffective classes.
A divisor class is universally pseudoeffective if and only if it is nef,
and a curve class is universally pseudoeffective if and only if it is in the closure of the cone of movable curves. 

We propose a general conjecture concerning
$2$-dimensional realizable classes on a smooth complex projective variety $X$. The idea is that $\mathscr{R}_2(X)$ has a predictable shape within the cone of universally pseudoeffective cycles. Set $\H^{1,1}(X)_{\mathbb{R}}\coloneq \H^2(X,\RR) \cap \H^{1,1}(X)$.

\begin{conjecture}\label{con:main}
    The following are equivalent for any universally pseudoeffective 
     $\eta\in \mathrm{B}_2(X)$.
    \begin{enumerate}[(1)]\itemsep 5pt
        \item The class $\eta$ is realizable over $\mathbb{R}$.
        \item 
        The matrix representation of the bilinear form
        \[
            \H^{1,1}(X)_{\mathbb{R}}\times \H^{1,1}(X)_{\mathbb{R}}  \longrightarrow \mathbb{R}, \qquad
            (\alpha,\beta) \longmapsto \int_{\eta}\alpha\wedge\beta
        \]
        has at most one positive eigenvalue.
    \end{enumerate}
\end{conjecture}

That  ($1$) implies  ($2$) follows from 
the Hodge--Riemann relations for $\H^{1,1}(S)$, where $S \to X$ is any map from a smooth projective surface. 
Since every Lorentzian matrix is a limit of Lorentzian matrices with rational entries, Theorem~\ref{thm:nondegenerateSteenrod} confirms Conjecture~\ref{con:main} when $X$ is a product of projective spaces.

The following example demonstrates that, in Conjecture \ref{con:main}, it is essential to require that $\eta$ is universally pseudoeffective rather than merely pseudoeffective. 

\begin{example}\label{ex:UPseudoNecessary}
Let $\pi: X\to \PP^3$ be the blowup of $\PP^3$ along a line $L$,
and consider the class
\[
\eta(a,b)=a[H]+b[E] \in \H_4(X,\ZZ),
\]
where $H$ is the pullback of a hyperplane and $E$ is the exceptional divisor. The Chow ring of $X$ is determined by the relations
\[
\int_X H^3=1, \ \ \int_X H^2E=0, \ \  \int_X HE^2=-1, \ \ \int_X E^3=-2, \ \ \text{and} \ \ (H-E)^2=0.
\]
Using these, it is straightforward to check that the bilinear form of $\eta(a,b)$ has at most one positive eigenvalue on $H^2(X,\RR)$ for any $a$ and $b$.
On the other hand, $\eta(a,b)$ is realizable over $\ZZ$ if and only if
\[
\Big(\text{$a>0$ and $b \le 0$ and $a+b>0$}\Big) \ \ \text{or} \ \ (a,b)=(1,-1) \ \ \text{or} \ \ (a,b)=(0,1).
\]
Indeed, if $b>0$, any fiber over a point in $L$ intersects $\eta(a,b)$ negatively, so an effective cycle with class $\eta(a,b)$ should contain  $E$ as a component. 

\end{example}

The following example demonstrates that a universally pseudoeffective  $\eta\in\H_4(X, \ZZ)$ satisfying the condition ($2$) in Conjecture \ref{con:main} need not be realizable over $\QQ$.

\begin{example}[Mumford's example]
Take a smooth projective surface $Y$ and a divisor $D$ on $Y$ with self-intersection zero such that the intersection of $D$ and any irreducible curve in $Y$ is positive \cite[Example 1.5.2]{Lazarsfeld1}. 
Let $Z$ be a smooth projective curve, and
take the nef class $\eta$ on $X \coloneq Y \times Z$ given by $\pi^*(D)$, 
where $\pi$ is the projection $X \to Y$.
As any nef class is a limit of ample classes \cite[Section 1.4]{Lazarsfeld1}, the class $\eta$ is realizable over $\RR$. 

We check that  $\eta$ is not realizable over $\QQ$.
 Suppose $\lambda \pi^*(D)$ is the class of an irreducible surface $S$ for some positive rational number $\lambda$.
 By the projection formula, the intersection of $S$ with a fiber of $\pi$ has degree zero, so $S$ must be of the form $\pi^{-1}(C)$ for some irreducible curve $C$ in $Y$. Since the intersection of $D$ and $C$ in $Y$ is positive, the intersection of $\pi^*(D)$ and $S$ is homologous to a positive sum of fibers of $\pi$, contradicting that the self-intersection of $D$ is zero.
\end{example}

We give an explicit universally pseudoeffective $\eta \in \H_4(X,\ZZ)$ for a $6$-dimensional smooth projective toric variety $X$ for which the validity of  Conjecture~\ref{con:main} is unknown.
See  \cite{HuhCDM} for background and unexplained terms.
We recall that, if $X$ is a toric variety, or more generally a spherical variety, then a class is universally pseudoeffective if and only if it is nef \cite[Example 4.5]{Fulger-Lehmann}.

\begin{example}[Matroid classes]


Let 
$X_n$ be the complex smooth projective toric variety of the $n$-dimensional permutohedron. 
The Bergman fan of any simple rank $3$ matroid $\mathrm{M}$ on $n+1$ elements defines a class $\Delta_\mathrm{M}\in \H_4(X_n, \ZZ)$ that is nef and effective \cite[Section 5]{HuhCDM}.
The Hodge--Riemann relations for $\mathrm{M}$ in \cite[Theorem 1.4]{JuneKarim} shows that $\Delta_\mathrm{M}$ satisfies the condition ($2$) of Conjecture~\ref{con:main}. 
If the matroid $\mathrm{M}$ is realizable over $\CC$, then $\Delta_\mathrm{M}$ is realizable over $\mathbb{Z}$ \cite[Section 6]{HuhCDM}.
Yu observes in \cite{JosephineYu} that, if the matroid $\mathrm{M}$ is not realizable over $\CC$, then
$\Delta_\mathrm{M}$ is not realizable over $\mathbb{Q}$.
Take any $\mathrm{M}$ that is not realizable over $\CC$. For example, take the matroid of the Fano plane $\PP^2(\mathbb{F}_2)$ on seven elements.
Is $\Delta_\mathrm{M}$ realizable over $\RR$?
\end{example}

When $X$ is a homogeneous space $G/P$ of Picard rank $1$, then every effective class is universally pseudoeffective, and the condition ($2$) of Conjecture~\ref{con:main} is automatically satisfied for all effective classes.
We verify  Conjecture~\ref{con:main} for Grassmannians by showing that
any effective class in $\H_4(\text{Gr}(k, n), \ZZ)$ is realizable over $\ZZ$, unless it is a multiple of some Schubert class. 

We suppose $2\leq k\leq n-2$, since otherwise the assertion is trivial.
Fix subspaces 
\[
F_i\coloneq \text{span}(\mathbf{e}_1,\mathbf{e}_2,\ldots,\mathbf{e}_i) \subseteq \mathbb{C}^n \ \  \text{and} \ \ 
G_i \coloneq \text{span}(\mathbf{e}_n,\mathbf{e}_{n-1},\ldots,\mathbf{e}_{i+1}) \subseteq \mathbb{C}^n,
\]
and consider the corresponding Schubert varieties in $\text{Gr}(k, n)$:
\begin{align*}
&S_1\coloneq\{\text{$k$-dimensional subspaces $V \subseteq \CC^n$ such that $F_{k-2}\subseteq V\subseteq F_{k+1}$}\}, \\
&S_2\coloneq \{\text{$k$-dimensional subspaces $V \subseteq \CC^n$ such that $F_{k-1}\subseteq V\subseteq F_{k+2}$}\},\\
&T_1\coloneq\{\text{$k$-dimensional subspaces $V \subseteq \CC^n$ such that $\dim V \cap G_{k-1}\ge 2$}\}, \\
&T_2\coloneq \{\text{$k$-dimensional subspaces $V \subseteq \CC^n$ such that $\dim V \cap G_{k+1} \ge 1$}\}.
\end{align*}
The classes of $S_1$ and $S_2$ form a $\ZZ$-basis of $\H_4(\text{Gr}(k, n),\ZZ)$.
For integers $a_1$ and $a_2$, we set
\[
\eta(a_1,a_2)\coloneq a_1 \, [S_1] + a_2 \,  [S_2].
\]
The intersections between cycles of dimension $2$ and codimension $2$ are given by
\[
[S_1] \cdot [T_1]=[S_2] \cdot [T_2]=[\text{point}] \ \  \text{and} \ \  [S_1] \cdot [T_2]=[S_2] \cdot [T_1]=0.
\]


\begin{theorem}\label{thm:Grassmannian}
The  class $\eta(a_1,a_2)$ is realizable over $\ZZ$ 
if and only if one of the following holds:
\begin{enumerate}[(1)]\itemsep 5pt
        \item  $a_1>0$ and $a_2>0$, or $(a_1,a_2)=(1,0)$ or $(a_1,a_2)=(0,1)$, when $k=2$ and $n=4$.
        \item $a_1\geq 0$ and $a_2>0$, or $(a_1,a_2)=(1,0)$, when $k=2$ and $n>4$.
        \item $a_2\geq 0$ and $a_1>0$, or $(a_1,a_2)=(0,1)$, when $k=n-2$ and $n>4$.
        \item $a_1,a_2\geq 0$ and $(a_1,a_2)\neq (0,0)$, when $n-2>k>2$.    
\end{enumerate}    
\end{theorem}

Compare Theorem~\ref{thm:Grassmannian} with the conjectural bound $a_1 \le 3a_2$ for homology classes of smooth surfaces in $\text{Gr}(2,4)$ with $a_1, a_2>0$ in
\cite{Dolgachev-Reider}. In \cite{Gross}, Gross proves  that the following holds for the class $\eta(a_1,a_2)$ of  any smooth surface $S \subseteq \text{Gr}(2,4)$ with $a_1, a_2>0$:
If $S$ is not of general type or $a_2 \le 19$, then $a_1 \le 3a_2$.

\begin{proof}
Suppose that $\eta(a_1,a_2)$ is realizable over $\ZZ$. We must have $a_1,a_2\geq 0$ and $(a_1,a_2) \neq (0,0)$ because the effective cone of $\text{Gr}(k, n)$ is generated by $[S_1]$ and $[S_2]$. 

We first show that $\eta(a_1,0)$ is realizable over $\ZZ$ in $\text{Gr}(2,n)$ only if $a_1=1$. By Grassmannian duality, this implies that $\eta(0,a_2)$ is realizable over $\ZZ$ in $\text{Gr}(n-2,n)$ only if $a_2=1$. These results follow from, for example, the characterisation of Schur rigid classes in \cite{robles2012rigid}. 
For the sake of completeness, we provide a direct argument.
Suppose $Y$ is an irreducible surface in $\text{Gr}(2,n)$  with $[Y]=\eta(a_1,0)$. 
Consider the union of projective lines
\[
Z=\bigcup_{V \in Y} \PP(V) \subseteq \PP(\CC^n).
\]
Since $Z$ is the image of a projective bundle over $Y$, it is irreducible.
We use Kleiman's generic transversality \cite{Kleiman} to extract information on $Z$ from $[Y]$:
\begin{enumerate}[(1)]\itemsep 5pt
\item The intersection of  $Y$ and a general translate  of $T_2$  is empty. Thus, there is a codimension $3$ linear subspace of $\PP(\CC^n)$ that does not intersect $Z$. Therefore, $Z$ is a surface.
\item The intersection of $Y$ and a general translate of $T_1$ consists of $a_1$ distinct points. Thus, the intersection of $Z$ with a general hyperplane of $\PP(\CC^n)$ contains $a_1$ distinct lines.
\end{enumerate}
By Bertini's theorem 
\cite[Theorem 7.1]{jouanolou1983}, this general hyperplane section of $Z$ is an irreducible curve, so $a_1$ is necessarily $1$.

We next show that $\eta(a_1,0)$ is realizable over $\ZZ$ in $\text{Gr}(k,n)$ for all $a_1>0$ when $k>2$.
By Grassmannian duality, this implies that $\eta(0,a_2)$ is realizable over $\ZZ$ in $\text{Gr}(k,n)$ for all $a_2>0$ when $n-2>k$.
Suppose $k>2$, and consider the Schubert variety
\[
\PP^k \simeq \{\text{$k$-dimensional subspaces $V \subseteq \CC^n$ with $ V\subseteq F_{k+1}$}\} \subseteq \text{Gr}(k,n).
\]
Since $k>2$, there is an irreducible surface $Y$ of degree $a_1$ in this $\PP^k$. Clearly, $[Y]=\eta(a_1,0)$.

    When $a_1,a_2>0$,
     the realizability can be reduced to the case when $n=4$ and $k=2$: 
     We have
\[
\text{Gr}(2,4)\simeq \{\text{$k$-dimensional subspaces $V \subseteq \CC^n$ with $F_{k-2}\subseteq V\subseteq F_{k+2}$}\} \subseteq \text{Gr}(k,n),
\]
 and the pushforward map of the inclusion sends $\eta(a_1,a_2)$ to $\eta(a_1,a_2)$.

We construct an explicit irreducible surface $Y$ in $\text{Gr}(2, 4)$ whose homology class is $\eta(a_1,a_2)$ for given positive integers $a_1$, $a_2$. 
By Grassmannian duality, we may assume that $a_1\geq a_2$. Consider the subspace
\[
V(x,y)\coloneq \text{span}\Big(\mathbf{e}_1+x^{a_1} \mathbf{e}_3+x^{a_2}\mathbf{e}_4,\mathbf{e}_2+y\mathbf{e}_3+x\mathbf{e}_4 \Big) \subseteq \CC^4,
\]
and let $Y$ be the closure in $\text{Gr}(2, 4)$ of the subset
\[
\Big\{ V(x,y), \  x, y\in \CC
\Big\} \subseteq \text{Gr}(2, 4).
\]
By Poincar\'e duality, it is enough to show that $[T_1] \cdot [Y]=a_1$ and $[T_2] \cdot [Y]=a_2$.

We compute the intersection of $Y$ with $T_1$. This amounts to
counting the number of subspaces $V(x,y)$ that are perpendicular to a given general vector $(c_1,c_2,c_3,c_4)$, that is, counting the number of solutions of the system of equations
\[
c_1+c_3x^{a_1}+c_4x^{a_2}=0 \quad \text{and} \quad c_2+c_3y+c_4x=0.
\]
Since $a_1 \ge a_2>0$, the first equation has $a_1$ solutions, and the second equation uniquely determines $y$ given $x$.
 This gives $[T_1] \cdot [Y]=a_1$.

We compute the intersection product of $Y$ with $T_2$. 
This amounts to counting the number of subspaces $V(x,y)$ that contain a given general vector $(c_1,c_2,c_3,c_4)$, that is, counting the number of solutions of the system of equations
\[
    c_1x^{a_1}+c_2y=c_3 \quad \text{and} \quad 
    c_1x^{a_2}+c_2x=c_4.
\]  
Since $a_2>0$, the second equation has $a_2$ solutions,
 and the first equation uniquely determines $y$ given $x$.
 This gives $[T_2] \cdot [Y]=a_2$.
\end{proof}

\bibliography{biblio}
\bibliographystyle{amsalpha}
\end{document}